\documentclass[a4paper,11pt]{amsart}

\usepackage{amsfonts,latexsym,rawfonts,amsmath,amssymb,amsthm, mathrsfs, lscape}
\usepackage[all]{xy}
\usepackage[english]{babel}
\usepackage[utf8]{inputenc}

\usepackage{array, tabularx}

\usepackage{setspace}
\usepackage{textcomp}

\usepackage[hypertexnames=false,
backref=page,
    pdftex,
    pdfpagemode=UseNone,
    breaklinks=true,
    extension=pdf,
    colorlinks=true,
    linkcolor=blue,
    citecolor=blue,
    urlcolor=blue,
]{hyperref}

\usepackage[top=1.5in, bottom=.9in, left=0.9in, right=0.9in]{geometry}

\newcommand{\referenza}{}

\newtheorem{thm}{Theorem}[section]
\newtheorem*{thm*}{Theorem \referenza}
\newtheorem{cor}[thm]{Corollary}
\newtheorem*{cor*}{Corollary \referenza}
\newtheorem{lem}[thm]{Lemma}
\newtheorem*{lem*}{Lemma \referenza}

\newtheorem{prop}[thm]{Proposition}
\newtheorem*{prop*}{Proposition \referenza}

\newtheorem{conj}[thm]{Conjecture}
\newtheorem*{conj*}{Conjecture \referenza}
\newtheorem{rmk}[thm]{Remark}

\newtheorem{exa}[thm]{Example}
\newtheorem{defi}[thm]{Definition}

\numberwithin{equation}{section}

\def \N {\mathbb N}

\def \R {\mathbb R}
\def \C {\mathbb C}

\def \P {\mathbb P}

\newcommand{\sspace}{\text{\--}}
\newcommand{\ssspace}{\text{\textdblhyphen}}

\allowdisplaybreaks[1]

\title{On Chern-Yamabe problem}

\author{Daniele Angella}
\address[Daniele Angella]{Centro di Ricerca Matematica ``Ennio de Giorgi''\\
Collegio Puteano, Scuola Normale Superiore\\
Piazza dei Cavalieri 3\\
56126 Pisa, Italy
}
\email{daniele.angella@sns.it}
\email{daniele.angella@gmail.com}

\author{Simone Calamai}
\address[Simone Calamai]{Dipartimento di Matematica e Informatica ``Ulisse Dini''\\
Università di Firenze\\
via Morgagni 67/A\\
50134 Firenze, Italy}
\email{scalamai@math.unifi.it}
\email{simocala@gmail.com}

\author{Cristiano Spotti}
\address[Cristiano Spotti]{Department of Pure Mathematics and Mathematical Statistics\\
Centre for Mathematical Sciences\\
University of Cambridge\\
Wilberforce Road, CB3 0WB\\
Cambridge, United Kingdom
}
\email{c.spotti@dpmmms.cam.ac.uk}

\thanks{During the preparation of the work, the first author has been granted by a research fellowship by Istituto Nazionale di Alta Matematica INdAM and by a Junior Visiting Position at Centro di Ricerca ``Ennio de Giorgi''; he is also supported by the Project PRIN ``Varietà reali e complesse: geometria, topologia e analisi armonica'', by the Project FIRB ``Geometria Differenziale e Teoria Geometrica delle Funzioni'', by SNS GR14 grant ``Geometry of non-Kähler manifolds'', and by GNSAGA of INdAM. The second author is supported by the Project PRIN ``Varietà reali e complesse: geometria, topologia e analisi armonica'', by  the Simons Center for Geometry and Physics, Stony Brook University, by SNS GR14 grant ``Geometry of non-Kähler manifolds'', and by GNSAGA of INdAM. During the preparation of this note, the third author has been supported by the grant ANR-10-BLAN 0105 at the ENS Paris and by a Post-doctoral bursary in the framework of Dr. Julius Ross EPSRC Career Acceleration Fellowship (EP/J002062/1).}

\date{\today}

\dedicatory{
Dedicated to Professor Paul Gauduchon on the occasion of his 70th birthday.\\
Buon compleanno!
}

\begin{document}

\begin{abstract}
 We initiate the study of an analogue of the Yamabe problem for complex manifolds. More precisely, fixed a conformal Hermitian structure on a compact complex manifold, we are concerned in the existence of metrics with constant Chern scalar curvature. In this note, we set the problem and we provide a positive answer when the expected constant Chern scalar curvature is non-positive. In particular, this includes the case when the Kodaira dimension of the manifold is non-negative.
 Finally, we give some remarks on the positive curvature case, showing existence in some special cases and the failure, in general, of uniqueness of the solution.
\end{abstract}

\maketitle

\section*{Introduction}

In this note, as a tentative to study special metrics on complex (possibly non-K\"ahler) manifolds, we investigate the existence of Hermitian metrics having constant scalar curvature with respect to the Chern connection.
A first natural basic question to ask, motivated by the classical Yamabe problem, is whether such metrics can actually exist in any (Hermitian) conformal class. We are going to show, via standard analytical techniques, that the answer is indeed affirmative in many situations. On the other hand, definitely not too surprisingly, things become more subtle in the case when the expected value of the constant Chern scalar curvature is positive, which can happen only when the Kodaira dimension of the manifold is equal to $-\infty$. We hope to discuss more systematically this last case in subsequent work.

\medskip

Let us now introduce more carefully our problem, giving some contextualization.
The solution of the Yamabe problem for compact differentiable manifolds (see, e.g., \cite{yamabe, trudinger, schoen, aubin-book, lee-parker}) assures the existence of a special Riemannian metric in every conformal class, characterized by having constant scalar curvature. In K\"ahler geometry, one seeks for K\"ahler metrics with constant scalar curvature in given cohomology classes, hoping to find necessary and sufficient conditions for their existence. On the other hand, for general complex non-K\"ahler manifolds, the {\itshape qu\^ete} for special metrics is somehow more mysterious. The abundance of Hermitian metrics leads to restrict to special metrics characterized by ``cohomological'' conditions. A foundational result in this sense is the theorem proven by P. Gauduchon in \cite{gauduchon-cras1977}. It states that, on a compact complex manifold of complex dimension $n\geq2$, every conformal class of Hermitian metrics contains a {\em standard}, also called {\em Gauduchon}, metric $\omega$, that is, satisfying $\partial\overline\partial\,\omega^{n-1}=0$. 
On the other hand, one can look instead at metrics with special ``curvature'' properties, related to the underlying complex structure. The focus of the present note is exactly on this second direction. In particular on the property of having \emph{constant Chern scalar curvature} in fixed conformal classes (hence neglecting cohomological conditions, for the moment). We should remark that this goes in different direction with respect to both the classical Yamabe problem and the Yamabe problem for almost Hermitian manifolds studied by H. del Rio and S. Simanca in \cite{delrio-simanca} (compare Remark \ref{Comparison}).
One motivation for considering exactly such ``complex curvature scalar'', between other natural ones \cite{liu-yang, liu-yang-2}, comes from the importance of the Chern Ricci curvature in non-K\"ahler Calabi-Yau problems (compare, for example, \cite{tosatti-weinkove}). We also stress that it would be very interesting, especially in view of having possibly ``more canonical'' metrics on complex manifolds, to study the problem of existence of metrics satisfying both cohomological and curvature conditions (e.g., Gauduchon metrics with constant Chern scalar 
curvature). 

\medskip

We now describe the problem in more details. Let $X$ be a compact complex manifold of complex dimension $n$ endowed with a Hermitian metric $\omega$. Consider the {\em Chern connection}, that is,  the unique connection on $T^{1,0}X$ preserving the Hermitian structure and whose part of type $(0,1)$ coincides with the Cauchy-Riemann operator associated to the holomorphic structure. The \emph{Chern scalar curvature} can be succinctly expressed as
$$ S^{Ch} (\omega) = \mathrm{tr}_\omega i\overline\partial\partial \log \omega^n,$$
where $\omega^n$ denotes the volume element.

Denote by $\mathcal{C}^H_X$ the space of Hermitian conformal structures on $X$.
On a fixed conformal class, there is an obvious  action of the following ``gauge group'':
$$ \mathcal{G}_X(\{\omega\}) := \mathcal{H}\mathrm{Conf}(X,\{\omega\}) \times \R^{+}, $$
where $\mathcal{H}\mathrm{Conf}(X,\{\omega\})$ is the group of biholomorphic automorphisms of $X$ preserving the conformal structure $\{\omega\}$ and $\R^+$ the scalings. It is then natural to study the moduli space
$$ \mathcal{C}h\mathcal{Y}a(X,\{\omega\}) :=  \left. \left\{\omega{'} \in \{\omega\} \; \middle|\; S^{Ch}(\omega{'}) \text{ is constant}\right\} \middle/ \mathcal{G}_X(\{\omega\}) \right. . $$
In analogy with the classical Yamabe problem, it is tempting to ask whether in each conformal class there always exists at least one metric having constant Chern scalar curvature. That is, one can ask whether the following Chern-Yamabe conjecture holds.

\renewcommand{\referenza}{\ref{conj:cy}}
\begin{conj*}[Chern-Yamabe conjecture]
 Let $X$ be a compact complex manifold of complex dimension $n$, and let $\{\omega\} \in \mathcal{C}^H_X$ be a Hermitian conformal structure on $X$. Then
 $$ \mathcal{C}h\mathcal{Y}a(X,\{\omega\}) \neq \varnothing \;. $$
\end{conj*}

By noting that, under conformal transformations, the Chern scalar curvature changes as
$$ S^{Ch} \left( \exp(2f/n) \omega \right) = \exp(-2f/n) \left( S^{Ch} \left(  \omega \right) + \Delta^{Ch}_\omega f \right), $$
where $\Delta^{Ch}_\omega$ denotes the Chern Laplacian with respect to $\omega$, that is,  $$\Delta^{Ch}_\omega f := \left( \omega ,\, dd^c f \right)_\omega = 2i\mathrm{tr}_\omega \overline\partial\partial f=\Delta_d f + (df,\, \theta )_\omega,$$
where $\theta=\theta(\omega)$ is the torsion $1$-form defined by $d\omega^{n-1} = \theta\wedge\omega^{n-1}$, 
the above conjecture is reduced in finding a solution $\left( f, \lambda \right) \in \mathcal{C}^\infty(X;\R) \times \R$ of the Liouville-type equation, which generalizes the classical equation for the Uniformization Theorem,
\begin{equation}\tag{\ref{eq:cy}}
 \Delta^{Ch}f + S = \lambda \exp(2f/n) \;,
\end{equation}
where $S:=S^{Ch}(\omega)$ and $\Delta^{Ch}:=\Delta^{Ch}_\omega$. In this way, the metric $\exp(2f/n)\omega\in\{\omega\}$ has constant Chern scalar curvature equal to $\lambda$. For a nice discussion on a related type of equations see \cite{malchiodi}.

\medskip

As one can expect, the difficulty of the above problem can be related to the ``positivity'' of certain invariant concerning curvature. Thanks to the Gauduchon theorem in \cite{gauduchon-cras1977}, one can introduce the following quantity. Assume that $n=\dim_\C X\geq2$. Given $\{\omega\}\in\mathcal{C}^H_X$, let $\eta\in\{\omega\}$ be the unique Gauduchon representative of volume $1$. Define the {\em Gauduchon degree} as the invariant of the conformal class given by
$$ \Gamma_X(\{\omega\}) := \frac{1}{(n-1)!} \int_X c_1^{BC}(K_X^{-1})\wedge\eta^{n-1} = \int_X S^{Ch}(\eta) d\mu_\eta \;.$$

Under a suitable normalization of the conformal potential, we have that $\Gamma_X(\{\omega\})$ is indeed exactly equal to the value of the eventual constant Chern scalar curvature, see Proposition \ref{prop:gamma-sign}. Furthermore, its sign is related to other invariants of the complex structure. For example, thanks to \cite{gauduchon-cras1981}, positivity of the Kodaira dimension implies that the Gauduchon degree of any conformal class is negative, see Proposition \ref{prop:kod-gauddegree}. Note that the positivity of the Kodaira dimension is sufficient but not necessary for the eventual constant Chern scalar curvature being negative: e.g., consider the Inoue surfaces in class VII, compare \cite{teleman} and Example \ref{ex:inoue}.

\medskip

At present, we are able to solve the Chern-Yamabe conjecture at least for Hermitian conformal classes having non-positive Gauduchon degree. In the case of zero Gauduchon degree, this follows from the fact that the equation \eqref{eq:cy} reduces to an easy linear elliptic equation. In the case of negative Gauduchon degree, we apply a continuity method to get a solution, by an easy extension of the techniques used in the standard analytic proof of the Uniformization Theorem for surfaces of genus $g\geq 2$. More precisely, we get the following.

\renewcommand{\referenza}{\ref{thm:cy-zero-gaud} and Theorem \ref{thm:cy-neg-gaud}}
\begin{thm*}
 Let $X$ be a compact complex manifold and let $\{\omega\}\in \mathcal{C}^H_X$.
 If $\Gamma_X(\{\omega\}) \leq 0$, then
 $\mathcal{C}h\mathcal{Y}a(X,\{\omega\}) =\{p\}$. Moreover, the constant Chern scalar curvature metric is ``strongly'' unique, i.e., unique up to scalings.
\end{thm*}

In particular, thanks to the relations between the Gauduchon degrees and the Kodaira dimension, we have the following corollary.

\renewcommand{\referenza}{\ref{cor:kod-non-negative}}
\begin{cor*}
The Chern-Yamabe Conjecture \ref{conj:cy} holds for any compact complex manifold $X$ having $\mathrm{Kod}(X)\geq 0$. 
\end{cor*}

Not too surprisingly, the study of the space $\mathcal{C}h\mathcal{Y}a(X,\{\omega\})$ in the positive case, even its eventually being non-empty, seems to definitely need more advanced analytical tools and geometric considerations.
We hope to treat in details this missing case in subsequent work. Here we just point out some easy observations.

In fact, in some cases as well as in some examples, it is definitely possible to have (non-K\"ahler) metrics with positive constant Chern scalar curvature, at least in some Hermitian conformal classes. For example, we can prove the following.

\renewcommand{\referenza}{\ref {prop:pos-product-curve}}
\begin{prop*}
Let $X$ be a compact complex manifold of complex dimension $\dim_\C X \geq2$ endowed with a Hermitian conformal class $\{\omega\}$ with $\Gamma_X(\{\omega\})>0$. Then $X\times \Sigma_g$, with $\Sigma_g$ compact complex curve of genus $g\geq2$, admits  Hermitian conformal classes with metrics having positive constant Chern scalar curvature.
\end{prop*}

Another fact to notice is that, if we restrict to conformal classes admitting a {\em balanced} representative, (that is, a Hermitian metric $\eta$ such that $d\eta^{n-1}=0$, where $n$ denotes the complex dimension of the manifold, \cite{michelsohn},) then the equation \eqref{conj:cy} is an Euler-Lagrange equation for a natural functional. In fact, this assumption is necessary in order to let the equation be the Euler-Lagrange equation with multiplier for the standard $L^2$ pairing, see Proposition \ref{prop:variational-form}. The functional is:
\begin{equation}\tag{\ref{eq:el-functional}}
\mathcal{F}(f) := \frac 1 2 \int_X |d f|_\eta^2 d\mu_\eta + \int_X S^{Ch}(\eta) f d\mu_\eta \;,
\end{equation}
with constraint
\begin{equation}\tag{\ref{eq:el-constraint}}
\int_X \exp(2f/n)d\mu_\eta=1\;.
\end{equation}
(Note that the constraint is not of having unitary volume.)

A possible strategy to attack the existence problem may be via the Chern-Yamabe flow, see Definition \ref{CYF}. However, we remark that the analytical challenging of Chern-Yamabe problem in positive Gauduchon degrees may suggests that the naive Chern-Yamabe conjecture is not always true in this case. If this happens, it will be still interesting  to understand what are the possible obstructions and to find some situations when a solution, or many solutions, actually exist(s), maybe giving interesting moduli spaces.  
Regarding this last point, we notice that, in general, solutions to the Chern-Yamabe problem are not unique in case of positive Gauduchon degree.

\renewcommand{\referenza}{\ref{prop:non-uniq-positive}}
\begin{prop*}
 In general, on a compact complex manifold $X$, constant Chern scalar curvature metrics in a fixed conformal class $\{\omega\}$ with positive Gauduchon degree are not unique, even up to the action of $\mathcal{G}_X(\{\omega\})$.
\end{prop*}

In other words, $\mathcal{C}h\mathcal{Y}a(X,\{\omega\}) \supsetneq \{p\}$.
This follows by the bifurcation techniques via odd crossing numbers, see \cite[Theorem A]{westreich}.
(See also \cite[II.4]{kielhofer}. Compare also \cite{dLPZ}, founded on \cite{smoller-wasserman}: there, bifurcation theory is used to study multiplicity of solutions for the classical Yamabe problem.) Non-uniqueness phenomena already happen on $\C\P^1 \times\C\P^1\times\C\P^1$ in the natural conformal classes given by suitably scaled products of the round metrics.

\medskip

The structure of this note is as follows. In Section \ref{sec:prelim}, we recall the relevant definitions. In Section \ref{sec:cy}, we define precisely the Chern-Yamabe problem and we show some of the basic facts related to it. Sections \ref{sec:zero-gaud} and \ref{sec:neg-gaud} are devoted to the solution of the problem in case of Gauduchon degree being zero or negative respectively. In the final Section \ref{sec:pos-gaud}, we discuss some remarks in view of the study of the case of positive Gauduchon degree. For readers' convenience, we have included Appendix \ref{app:formulas} containing some formulas for the Chern Laplacian and the Chern scalar curvature.

\bigskip

\noindent{\sl Acknowledgments.}
The authors are grateful to Prof. X.~X. Chen, Prof. P. Gauduchon, Prof. A. Malchiodi, and Prof. P. Piccione for useful conversations, suggestions, and for their interest on the subject.
Part of this work has been developed during the 2014 SMI summer course by Paul Gauduchon in Cortona: the authors are grateful to the Professor, to the participants, and to the organizers of the school for the exceptional environment they contributed to.
Many thanks are also due to M. Zedda, D. Petrecca, S. Sun, P. Baroni, L. Cremaschi, K. Zheng for several discussions.

\section{Preliminaries and notation}\label{sec:prelim}
Let $X$ be a compact complex manifold of complex dimension $n$. We will usually confuse Hermitian metrics $h$ and their associated $(1,1)$-forms $\omega:=h(J\sspace,\ssspace)$. Given a Hermitian metric $\omega$, its volume form will be denoted by $d\mu_\omega$, and its total volume by $\mathrm{Vol}_\omega:=\int_X d\mu_\omega$.

The Hodge-de Rham Laplacian associated to $\omega$ will be denoted by $\Delta_{d,\omega}$. (Hereafter, the reference to the metric will not be specified when it is clear.) We assume the following sign convention: whenever $u$ is a smooth 
function on $X$ with real values, then at a point $p$ where $u$ attains its local maximum there holds $\left( \Delta_d u \right)(p) \geq 0$.
In these notation, $\Delta_d u = d^* d u$ on smooth functions $u$, where $d^*$ denotes the adjoint operator of $d$ with respect to $\omega$.
In particular, integration by part is written as
$$ \int_X  u \Delta_d v d\mu_\omega = \int_X (du , \, dv )_\omega d\mu_\omega \; ,$$
for $u,v$ smooth functions on $X$.

Once fixed a Hermitian structure $(J,\omega)$ on $X$, we denote by $\nabla^{Ch}$ the {\em Chern connection}, that is, the unique Hermitian connection on $T^{1,0}X$ (namely, $\nabla^{Ch}J=\nabla^{Ch} h =0$,) such that its part of type $(0,1)$ coincides with the Cauchy-Riemann operator $\overline\partial$ associated to the holomorphic structure; see, e.g., \cite{gauduchon-bumi}.
Denote by $T^{Ch}$ the torsion of $\nabla^{Ch}$.

\subsection{Special Hermitian metrics}
Suppose $n\geq2$. For a Hermitian metric $\omega$ on $X$, consider the operator
$$ L\colon \wedge^{\bullet}X\to\wedge^{\bullet+2}X \;, \qquad L:=\omega\wedge\sspace \;.$$
It satisfies that
$$ L^{n-1} \colon \wedge^1X \stackrel{\simeq}{\to} \wedge^{2n-1}X $$
is an isomorphism. In particular, there exists $\theta\in\wedge^1X$ such that
$$ d\omega^{n-1} = \theta\wedge\omega^{n-1} \;.$$
Such $\theta$ is called the {\em (balanced) Lee form}, or {\em torsion $1$-form}, associated to $\omega$. In terms of the torsion of the Chern connection, it can be given as follows, \cite[page 500]{gauduchon-mathann}. Take $\{ v_i \}_{i=1,\ldots, 2n }$ an $\omega$-orthonormal local frame of $TX$. Then
$$ \theta = \sum_{j=1}^{2n} g \left( T^{Ch}(\sspace , v_j ) ,\, v_j \right)\; .$$

The Hermitian metric $\omega$ is called {\em balanced} (in the sense of Michelsohn \cite{michelsohn}) if $\theta=0$, that is, $d\omega^{n-1}=0$. It is called {\em Gauduchon} (or {\em standard} in the notation of \cite{gauduchon-cras1977}) if $d^*\theta=0$. Note that the condition $d^*\theta=0$ is equivalent to $i\partial\overline\partial\omega^{n-1}=0$ (use that $\theta=\star_\omega d \omega^{n-1}$ where $\star_\omega=J*$ is the symplectic Hodge-$\star$-operator, see, e.g., \cite[\S1.10]{gauduchon-book}).

\subsection{Chern Laplacian} Fix a Hermitian metric $\omega$ on $X$, and consider the associated Chern Laplacian $\Delta^{Ch}:=\Delta^{Ch}_{\omega}$: on smooth functions $f$, it is defined as
$$ \Delta^{Ch} f = \left( \omega ,\, dd^c f \right)_\omega = 2i\mathrm{tr}_\omega\overline\partial\partial f \;, $$
or, in local holomorphic coordinates $\{z^j\}_j$ such that $\omega\stackrel{\text{loc}}{=}i h_{i\bar j}dz^i\wedge d\bar z^j$, as
$$ \Delta^{Ch} \stackrel{\text{loc}}{=} -2h^{i\bar j}\partial_i \partial_{\bar j} \;, $$
where $\left(h^{i\bar j}\right)_{i,j}$ denotes the inverse matrix of $\left(h_{i\bar j}\right)_{i,j}$.

Note that the Chern Laplacian is a Hopf operator, that is, a differential elliptic operator of $2$nd order without terms of order $0$. Its index is equal to the index of the Hodge-de Rham Laplacian. (See \cite[page 388]{gauduchon-cras1977}.)

As regards the comparison between the Hodge-de Rham Laplacian $\Delta_d$ and the Chern Laplacian $\Delta^{Ch}$ on smooth functions, the following holds (for the sake of completeness, we provide a proof in Appendix \ref{app:formulas}).

\begin{lem}[{\cite[pages 502--503]{gauduchon-mathann}}]\label{lem:chern-laplacian}
 Let $X$ be a compact complex manifold endowed with a Hermitian metric $\omega$ with Lee form $\theta$. The Chern Laplacian on smooth functions $f$ has the form
 $$ \Delta^{Ch}f = \Delta_d f + (df,\, \theta )_\omega \;. $$
\end{lem}

In particular, Chern Laplacian and Hodge-de Rham Laplacian on smooth functions coincide when $\omega$ is balanced.

\subsection{Hermitian conformal structures}
We will denote by $\mathcal{C}^H_X$ the {\em space of Hermitian conformal structures} on $X$.
Its elements are as follows: the Hermitian conformal class of $\omega$ will be denoted by
$$ \{\omega\} := \left\{ \exp({2f/n})\omega \;\middle|\; f\in\mathcal{C}^\infty(X;\R) \right\} \;. $$

We recall the foundational theorem by P. Gauduchon on the existence of standard metrics in Hermitian conformal classes.
\begin{thm}[{\cite[Théorème 1]{gauduchon-cras1977}}]
 Let $X$ be a compact complex manifold of complex dimension $\dim_\C X\geq2$, and fix a Hermitian conformal structure $\{\omega\}$.
 Then there exists a unique Gauduchon metric $\eta$ in $\{\omega\}$ such that $\int_X d\mu_\eta=1$.
\end{thm}

We will consider the following normalization: denote by $\eta\in\{\omega\}$ the unique Gauduchon representative of volume $1$, and set
$$ \{\omega\}_1 := \left\{ \exp(2f/n)\eta\in\{\omega\} \;\middle|\; \int_X \exp(2f/n)d\mu_\eta=1 \right\} \subset \{\omega\} \;. $$

\subsection{Underlying Gauge group}
The relevant {\em gauge group} to consider in our problem is the finite dimensional group of \emph{holomorphic conformal maps} and \emph{scalings}, that is,
$$ \mathcal{G}_X(\{\omega\}) := \mathcal{H}\mathrm{Conf}(X,\{\omega\}) \times \R^+ \subseteq \mathrm{Aut}(X)\times \R^+ \subseteq \mathrm{Diff}^+(X)  \times \R^+ \;, $$
where
$$ \mathcal{H}\mathrm{Conf}(X,\{\omega\}) := \left\{\phi \in \mathrm{Aut}(X) \; \middle| \; \phi^\ast \omega \in \{\omega\} \right\} \;. $$
The natural action of $\mathcal{G}_{X}(\{\omega\})$ on the conformal class $\{\omega\}$ is
$$(\phi, \lambda) . \left(\exp(2f/n)\omega\right) := \lambda \cdot \exp(2f\circ\phi/n)\phi^\ast \omega \;. $$

\medskip

We prove that, if $X \ncong \C\P^1$, then $\mathcal{H}\mathrm{Conf}(X,\{\omega\})$ is a compact Lie group in $\mathrm{Aut}(X)$, as a consequence of the uniqueness of the Gauduchon representative with volume normalization.
Thanks to this lemma, it follows also that, after fixing the Gauduchon representative $\eta$ of volume $1$ as reference point in the conformal class, the action of $\mathcal{H}\mathrm{Conf}(X,\{\omega\})\cong\mathcal{H}\mathrm{Isom}(X,\eta)$ on $\{\omega\}=\{\eta\}$ is just given by $\phi.\exp(2f/n)\eta=\exp\left(2f\circ\phi/n\right)\eta$. In particular, the subset $\{\omega\}_1$ is invariant for $\mathcal{H}\mathrm{Conf}(X,\{\omega\})$, and then a slice for the action $\mathcal{G}_X (\{\omega\}) \curvearrowright \{\omega\}$.

 \begin{lem}\label{lem:CGG} Let $X$ be a compact complex manifold of complex dimension $n\geq 2$, and let $\{\omega\}\in\mathcal{C}^H_X$.
  Then $\mathcal{H}\mathrm{Conf}(X,\{\omega\})$ is isomorphic to the compact Lie group of holomorphic isometries of the Gauduchon representative $\eta\in\{\omega\}$ of volume $1$, i.e.,
  $$\mathcal{H}\mathrm{Conf}(X,\{\omega\})\cong \mathcal{H}\mathrm{Isom}(X,\eta).$$
 \end{lem}
 
 \begin{proof}
  Let $\phi \in \mathcal{H}\mathrm{Conf}(X,\{\omega\})$. Since $\phi$ is holomorphic, we have
  $$0=\phi^\ast(i\partial \bar \partial \eta^{n-1})= i\partial \bar \partial (\phi^\ast \eta)^{n-1}\;,$$
  that is, $\phi^\ast \eta$ is also a Gauduchon representative in $\{\omega\}$ with volume $1$. But then, by the uniqueness of the Gauduchon metric with volume normalization, we get $\phi^\ast \eta= \eta$, that is, $\phi$ is a holomorphic isometry of $(X,\eta)$. The converse is obvious.
  
  Finally, to see that such a group is a compact subgroup of $\mathrm{Aut}(X)$, take a sequence $\left\{\phi_j\right\}_j \subset \mathcal{H}\mathrm{Isom}(X,\eta)$. By compactness of the space of Riemannian isometries, we can take a subsequence $\left\{\phi_{j_k}\right\}_k$ converging, in the $\mathcal{C}^\infty$ sense, to an isometry $\phi_{\infty}$ of $(X,\eta)$. However, since $\mathrm{Aut}(X)$ is closed in $\mathrm{Diff}(X)$, as a consequence of the Montel theorem, then $\phi_{\infty}$ is also holomorphic. That is $\phi_{\infty} \in \mathcal{H}\mathrm{Isom}(X,\eta)$ as claimed.
 \end{proof}
 
We should remark that the fact that the group of (holomorphic) conformal isometries is compact follows also from the classical theorems of M. Obata \cite{obata} and J. Lelong-Ferrand \cite{lelong-ferrand}.

As immediate consequence, we see that the fibres over $[\omega']$ of the quotient map $$\{\omega\}_1 \longrightarrow \{\omega\}_1 / \mathcal{H}\mathrm{Conf}(X,\{\omega\})$$ are \emph{compact homogeneous spaces} of the form $\mathcal{H}\mathrm{Isom}(X,\eta)/\mathcal{H}\mathrm{Isom}(X,\omega')$. 
 In particular, the holomorphic isometry group of the Gauduchon metric is the largest along the conformal class.

\subsection{Chern scalar curvature}
Once fixed a Hermitian structure $(J,\omega)$ on $X$, denote by $S^{Ch}$ the {\em Chern scalar curvature}, that is, the scalar curvature with respect to the Chern connection. 
We have the following result, for the proof of which we refer to Appendix \ref{app:formulas}.
\begin{prop}\label{prop: J parallel connection scalar curvature}
 Let $X$ be a $2n$-dimensional smooth manifold endowed with an almost complex structure $J$, and let $\nabla$ be an affine connection on $T^{1,0}X$, for which there holds $\nabla J =0$. Take $h$ a Hermitian metric on $X$. Fix $\{ Z_i \}_{i= 1 , \ldots , n}$ a local coordinate frame of vector fields on $T^{1,0}X$.
 Then the scalar curvature of $\nabla$ is given by
 $$ S^{\nabla} = h^{i\bar j} \left( \partial_i \Gamma_{\bar j k}^k - \partial_{\bar j} \Gamma_{ik}^k \right) \;.  $$
\end{prop}

In particular, for the Chern connection, we get the following well-known result.
\begin{prop}\label{prop:sch-local-coord}
 Let $X$ be a complex manifold of complex dimension $n$ endowed with a Hermitian metric $\omega$.
 In local holomorphic coordinates $\{z^j\}_j$ such that $\omega\stackrel{\text{loc}}{=}ih_{i\bar j}dz^i\wedge d\bar z^j$, the scalar curvature $S^{Ch}$ of the Chern connection $\nabla^{Ch}$ is given by
 $$ S^{Ch} (\omega) \stackrel{\text{loc}}{=} h^{i\bar j} \left( -\partial_i \partial_{\bar j} \log \det h_{k\bar \ell}  \right) = \mathrm{tr}_\omega i\overline\partial\partial\log\omega^n \; . $$
\end{prop}

\begin{proof}
Notice that, for the Chern connection, there holds
$$ \Gamma_{ij}^j = \partial_i \log \det h_{k\bar \ell} \;, $$
(see, e.g., \cite[Proposition 1.7.1]{gauduchon-book},)
and that $\Gamma_{\bar j k}^\ell = 0$. Thus, by means of Proposition \ref{prop: J parallel connection scalar curvature}, we get the statement.
\end{proof}

\section{Chern-Yamabe problem}\label{sec:cy}
Once fixed a Hermitian conformal structure $\{\omega\}$ on a compact complex manifold $X$, we are aimed at finding the ``best'' Hermitian metric. Here, we intend the problem in terms of constant Chern scalar curvature.

Note that, for $\omega'\in \{\omega\}$ and for $(\phi,\lambda) \in \mathcal{G}_X(\{\omega\})$, it holds $S^{Ch}( \lambda \phi^{\ast} \omega')= \frac{1}{\lambda} \phi^{\ast} S^{Ch}(\omega')$. Hence the condition of having constant Chern scalar curvature is invariant under the action of the gauge group. 
That is, the problem of finding metrics with constant Chern scalar curvature in the conformal class $\{\omega\}$ moves to the quotient $\{\omega\} / \mathcal{G}_X (\{\omega\})$.

\subsection{Chern-Yamabe conjecture}
 
 We define now the relevant moduli spaces for our Chern-Yamabe problem.
 Let $X$ be a compact complex manifold of complex dimension $n$, and let $\{\omega\} \in \mathcal{C}^H_X$. Then we define the \emph{moduli space of metrics with constant Chern scalar curvature} in the fixed conformal class to be the set
 $$ \mathcal{C}h\mathcal{Y}a(X,\{\omega\}) := \left. \widetilde{\mathcal{C}h\mathcal{Y}a(X,\{\omega\})} \middle/ \mathcal{H}\mathrm{Conf}(X,\{\omega\}) \right. \;, $$
 where
 $$ \widetilde{\mathcal{C}h\mathcal{Y}a(X,\{\omega\})} := \left\{\omega'\in \{\omega\}_1 \; \middle|\; S^{Ch}(\omega') \text{ is constant}\right\} \;. $$
 
 Then the Chern-Yamabe conjecture reads as follows.
 
 \begin{conj}[Chern-Yamabe conjecture]\label{conj:cy}
  Let $X$ be a compact complex manifold of complex dimension $n$, and let $\{\omega\} \in \mathcal{C}^H_X$ be a Hermitian conformal structure on $X$. Then
  $$ \mathcal{C}h\mathcal{Y}a(X,\{\omega\}) \neq \varnothing \;. $$
 \end{conj}
 
  \begin{rmk}\label{Comparison} We observe that the above Chern-Yamabe conjecture goes in different direction with respect to both the classical Yamabe problem and the Yamabe problem for almost Hermitian manifolds studied by H. del Rio and S. Simanca in \cite{delrio-simanca}. In fact, by \cite[Corollary 4.5]{liu-yang-2}, on a compact complex manifold, if the (average of the) Chern scalar curvature of a Hermitian metric is equal either to the (average of the) scalar curvature of the corresponding Riemannian metric or to the (average of the) $J$-scalar curvature of the corresponding Riemannian metric, then the metric has to be K\"ahler.
 \end{rmk}

\begin{exa}[Compact complex curves]
 On compact real surfaces, fixing the conformal class is the same as fixing the (necessary integrable) complex structure, since $CO(2)\cong GL(1,\C)$. Hence the conformal biholomorphic maps are all the biholomorphic maps of the complex curve $C$, that is, $\mathcal{H}\mathrm{Conf}(C,\{\omega\})=\mathrm{Aut}(C)$. Thus $\mathcal{C}h\mathcal{Y}a(\C\P^1)=\{p\}$ corresponding to the standard round metric. Note that the metric is unique only up to $\mathrm{Aut}(\C\P^1)\cong \mathrm{PGL}(2,\C)$. Thus the point $p$ still has a stabilizer equal to the isometries, i.e., $\mathrm{SO}(3)$. For any other complex curve $C$ of genus $g\geq 1$ one has again $\mathcal{C}h\mathcal{Y}a(C)=\{p\}$ corresponding to the unique metric $g_{\text{cst}}$ of constant (negative, if not torus) scalar curvature. In particular, if $g\geq 2$ the stabilizer is the discrete group $\mathrm{Aut}(C)\cong \mathrm{Isom}(C, g_{\text{cst}})$ and the metric is unique, even before taking quotient.
\end{exa}

\subsection{Gauduchon degree}
 From now on let us assume $\dim_\C X \geq 2$. Thanks to the foundational theorem by P. Gauduchon, we can introduce a natural \emph{invariant} of the conformal class $\{\omega\}\in\mathcal{C}^H_X$, namely, the \emph{Gauduchon degree}
 $$ \Gamma_X (\{\omega\}) \in \R $$
 defined as follows:
 take $\eta\in\{\omega\}$ the unique Gauduchon representative in $\{\omega\}$ of volume $1$, and define 
 $$ \Gamma_X(\{\omega\}) := \frac{1}{(n-1)!} \int_X c_1^{BC}(K_X^{-1})\wedge\eta^{n-1} = \int_X S^{Ch}(\eta) d\mu_\eta \;.$$
 It corresponds to the degree of the anti-canonical line bundle $K_X^{-1}$. Recall \cite{gauduchon-cras1981} (see also \cite[\S I.17]{gauduchon-mathann}) that the degree of a holomorphic line bundle is equal to the volume of the divisor associated to any meromorphic section by means of the Gauduchon metric.

\medskip

Finally we have the following result, after Gauduchon, concerning the relation between $\Gamma_X(\{\omega\})$ and the {\em Kodaira dimension} of $X$, that is,
$$ \mathrm{Kod}(X) \;:=\; \limsup_{m\to+\infty} \frac{\log\dim_\C H^0(X;K_X^{\otimes m})}{\log m} \;. $$
We recall that $K_X$ is said to be {\em holomorphically torsion} if there exists $\ell\in\N\setminus\{0\}$ such that $K_X^{\otimes\ell}=\mathcal{O}_X$; in this case, $c_1^{BC}(X)=0$.

\begin{prop}\label{prop:kod-gauddegree}
Let $X$ be a compact complex manifold.
If $\mathrm{Kod}(X) > 0 $, then $\Gamma_X(\{\omega\})<0$ for any Hermitian conformal class $\{\omega\} \in \mathcal{C}^H_X$. If $\mathrm{Kod}(X) = 0 $, then $\Gamma_X(\{\omega\})<0$ for any class $\{\omega\} \in \mathcal{C}^H_X$ beside the case when $K_X$ is holomorphically torsion, when $\Gamma_X(\{\omega\})=0$.  
\end{prop}

Few remarks are in place. When the canonical bundle is holomorphically torsion,  one get that the first Bott-Chern class vanishes, i.e.,  $c_1^{BC}(X)=0$.
Hence any conformal class actually carries a metric with vanishing Chern-Ricci curvature, by using the Chern-Ricci potential as a conformal potential; see \cite[Theorem 1.2]{tosatti}.
Moreover note that $\mathrm{Kod}(X)=-\infty$ does not imply that $\Gamma_X(\{\omega\})>0$. An example is given by the Inoue surface, see Example \ref{ex:inoue}. In fact, it holds the following: if there exists at least one non-trivial pluri-anti-canonical section $s \in H^0(X; K_X^{\otimes (-m)})$, then $\Gamma_X(\{\omega\})>0$ for all $\{\omega\}$ in $\mathcal{C}^H_X$.

\subsection{Semi-linear PDE for the Chern-Yamabe problem}
Let $X$ be a compact complex manifold of complex dimension $n$, and fix a Hermitian metric $\omega$ in the normalized Hermitian conformal class $\{\omega\}_1$.
In this section, we restate the Chern-Yamabe problem in terms of a semi-linear elliptic equation of $2$nd order.

\medskip

Consider a conformal change $\tilde \omega = \exp(2f/n) \omega$. 
By Proposition \ref{prop:sch-local-coord}, the Chern scalar curvature changes as follows:
\begin{equation}\label{eq:conf-change-sch}
S^{Ch} \left( \exp(2f/n) \omega \right) = \exp(-2f/n) \left( S^{Ch} \left(  \omega \right) + \Delta^{Ch}_\omega f \right)\; ,
\end{equation}
where we recall that $\Delta^{Ch}_\omega f = \Delta_{d,\omega} f + (df,\, \theta )_\omega$, see Lemma \ref{lem:chern-laplacian}.

Hence, the Chern-Yamabe Conjecture \ref{conj:cy} translates into the following.

\begin{conj}[Chern-Yamabe conjecture --- analytic equation]\label{conj:cy-2}
 Let $X$ be a compact complex manifold of complex dimension $n$ and let $\{\omega\} \in \mathcal{C}^H_X$ be a Hermitian conformal structure on $X$. Let $S:=S^{Ch}(\omega)$ be the Chern scalar curvature of a fixed representative $\omega\in\{\omega\}$. Then there exists $\left( f, \lambda \right) \in \mathcal{C}^\infty(X;\R) \times \R$ solution of the equation
 \begin{equation}\label{eq:cy}\tag{ChYa}
  \Delta^{Ch}f + S = \lambda \exp(2f/n) \;.
 \end{equation}
 
 We may also suppose that $f$ satisfies the following normalization. Let $\eta$ be the unique Gauduchon representative in $\{\omega\}$ with volume $1$. Let $\eta=\exp(-2g/n)\omega$. Then we assume that
 \begin{equation}\label{eq:cy-constraint}\tag{ChYa-n}
  \int_X \exp(2(f+g)/n)d\mu_\eta = 1 \;.
 \end{equation} 
\end{conj}
In fact, the metric $\exp(2f/n)\omega=\exp(2(f+g)/n)\eta\in\{\omega\}_1$ has constant Chern scalar curvature equal to $\lambda$, and so it yields a class in $\mathcal{C}h\mathcal{Y}a(X,\{\omega\})$.

\medskip

Concerning the sign of the expected constant Chern scalar curvature $\lambda$ as above, it is related to the Gauduchon degree as follows.
\begin{prop}\label{prop:gamma-sign}
 Let $X$ be a compact complex manifold and let $\{\omega\}\in\mathcal{C}^H_X$. Assume that $\omega'\in\{\omega\}_1$ has constant Chern scalar curvature equal to $\lambda\in\R$.
 Then
 $$ \Gamma_X(\{\omega\}) = \lambda \;. $$
 In particular, the sign of the Gauduchon degree $\Gamma_X(\{\omega\})$ is equal to the sign of the eventual constant Chern scalar curvature metric in the class $\{\omega\}$.
\end{prop}
 
\begin{proof}
 As representative in $\{\omega\}$, fix the unique Gauduchon metric $\eta\in\{\omega\}$ of volume $1$. Equation \eqref{eq:cy} yields
 $$ \int_X \Delta^{Ch}_\eta f d\mu_\eta + \int _X S^{Ch}(\eta) d\mu_\eta = \lambda \int_X \exp(2f/n) d\mu_\eta $$
 where (denote by $\theta$ the Lee form associated to $\eta$,)
 $$ \int_X \Delta^{Ch}_\eta f d\mu_\eta = \int_X \Delta_d f d\mu_\eta + \int_X (df, \theta) d\mu_\eta = \int_X \Delta_d f d\mu_\eta + \int_X (f, d^*\theta) d\mu_\eta = 0 $$
 since $\eta$ is Gauduchon.
 Therefore
 $$ \Gamma_X(\{\omega\}) = \int_X S^{Ch}(\eta) d\mu_\eta = \lambda \int_X \exp(2f/n) d\mu_\eta = \lambda \;, $$
 yielding the statement.
\end{proof}

 It is maybe interesting to point out a small remark on the relation between the above normalization and the constant volume one slice.
 By the Jensen inequality, it follows that the volume of metrics in the slice $\{\omega\}_1$ is always greater than or equal to $1$.
 This implies that, on the more geometric slice of volume one metrics, the value of the (non-zero) eventual constant Chern scalar curvature metrics is, in modulus, always greater than or equal to the Gauduchon degree.

\section{Solution of the Chern-Yamabe problem for zero Gauduchon degree}\label{sec:zero-gaud}
In case of zero Gauduchon degree, the semi-linear elliptic differential equation \eqref{eq:cy} becomes just linear, and so we get a solution for the corresponding Chern-Yamabe problem. In the following, we omit same analytical details concerning the precise function spaces where we realize the equation, since similar issues will be more carefully addressed in the next section when we will consider a non-linear situation. 

\begin{thm}\label{thm:cy-zero-gaud}
 Let $X$ be a compact complex manifold and let $\{\omega\}\in \mathcal{C}^H_X$.
 If $\Gamma_X(\{\omega\}) = 0$, then
 $\widetilde{\mathcal{C}h\mathcal{Y}a(X,\{\omega\})} =\{p\}$. In particular, denoting with $\omega_p\in\{\omega\}_1$ such unique metric with constant Chern scalar curvature, we have:
 \begin{itemize}
  \item $S^{Ch}(\omega_p)=\Gamma_X(\{\omega\})=0$;
  \item $\mathcal{H}\mathrm{Isom}(X,\omega_p)\cong \mathcal{H}\mathrm{Isom}(X,\eta) \cong \mathcal{H}\mathrm{Conf}(X,\{\omega\})$, where $\eta\in\{\omega\}_1$ denotes the unique Gauduchon representative in $\{\omega\}$ with volume $1$.
 \end{itemize}
\end{thm}

\begin{proof}
Fix $\eta\in\{\omega\}$ the unique Gauduchon representative in $\{\omega\}$ with volume $1$.
We are reduced to solve
$$ \Delta^{Ch} f = -S^{Ch}(\eta) \; . $$
Notice that the (formal) adjoint of $\Delta^{Ch}$ on smooth functions $g$ is 
$$ (\Delta^{Ch})^* g = \Delta_d g - (dg,\, \theta )_{\eta} \, , $$
where $\theta$ denotes the Lee form of $\eta$.
It has $\ker (\Delta^{Ch})^*$ equal to the constants.
Indeed, take $u$ in the kernel of $(\Delta^{Ch})^*$. We may assume $u$ is smooth by standard regularity. Then 
\begin{align}
 0= \int_X u (\Delta^{Ch})^* u d\mu_{\eta} = \int_X \left( |\nabla u|^2 - \frac{1}{2} (du^2 , \, \theta ) \right) d\mu_{\eta}  = 
\int_X  |\nabla u|^2  d\mu_{\eta} \, ,
\end{align}
since $d^*\theta=0$ because $\eta$ is Gauduchon.

Since the integral of  $-S^{Ch} (\eta)$ is zero, 
it means that $-S^{Ch} (\eta)\in \left( \ker (\Delta^{Ch})^* \right)^\perp = \mathrm{imm} \Delta^{Ch}$.
We thus conclude the existence of a metric of zero Chern scalar curvature. The uniqueness follows immediately by the fact that the kernel of the Chern Laplacian at the Gauduchon metric consists of just constants.

The last statements follow from Proposition \ref{prop:gamma-sign} and the assumption, and from Lemma \ref{lem:CGG} and the uniqueness.
\end{proof}

We should remark that the case of zero Gauduchon degree is genuinely interesting only when $\mathrm{Kod}(X)=-\infty$.

\section{Solution of the Chern-Yamabe problem for negative Gauduchon degree}\label{sec:neg-gaud}

In this section, we provide a first positive answer for the Chern-Yamabe conjecture in case of negative Gauduchon degree.

\begin{thm}\label{thm:cy-neg-gaud}
 Let $X$ be a compact complex manifold and let $\{\omega\}\in \mathcal{C}^H_X$.
 If $\Gamma_X(\{\omega\}) < 0$, then $\widetilde{\mathcal{C}h\mathcal{Y}a(X,\{\omega\})} =\{p\}$. In particular, denoting with $\omega_p\in\{\omega\}_1$ such unique metric with constant Chern scalar curvature, we have
 \begin{itemize}
  \item $S^{Ch}(\omega_p)=\Gamma_X(\{\omega\})<0$;
  \item $\mathcal{H}\mathrm{Isom}(X,\omega_p)\cong \mathcal{H}\mathrm{Isom}(X,\eta)\cong \mathcal{H}\mathrm{Conf}(X,\{\omega\})$, where $\eta\in\{\omega\}$ denotes the unique Gauduchon representative in $\{\omega\}$ with volume $1$.
 \end{itemize}
\end{thm}

\begin{proof}
Fix $\eta\in\{\omega\}$ the unique Gauduchon representative in $\{\omega\}$ with volume $1$. By assumption, we have
$$ \Gamma_X(\{\omega\}) = \int_X S^{Ch}(\eta) d\mu_\eta < 0 \;. $$

\medskip

A {\bfseries preliminary step} consists in showing that, without loss of generality, {\em we may assume that $S^{Ch}(\omega)<0$ at every point}.
In fact, we choose a representative in $\{\eta\}$ as follows. Fix $\eta$ as reference metric. Consider the equation
\begin{align}
 \Delta^{Ch}_\eta f = -S^{Ch} (\eta) + \int_X S^{Ch}(\eta) d\mu_{\eta} \: .
\end{align}
Since $\eta$ is Gauduchon, arguing as in the proof of Theorem \ref{thm:cy-zero-gaud}, the above equation has a solution $f\in\mathcal{C}^\infty(X;\R)$; such a solution is unique if we require $\int_X \exp(2f/n)d\mu_\eta=1$.
Then $\exp(2f/n)\eta\in\{\omega\}_1$ satisfies
\begin{eqnarray*}
 S^{Ch} (\exp(2f/n)\eta) &=&  \exp(-2f/n) \left(S^{Ch}(\eta)+\Delta^{Ch}f\right) \\[5pt]
 &=& \exp(-2f/n) \int_X S^{Ch}(\eta) d\mu_{\eta} \\[5pt]
 &=& \exp(-2f/n) \Gamma_X(\{\omega\}) < 0 \; .
\end{eqnarray*}

\medskip

The next step consists in applying, under this further assumption on the reference metric, the {\bfseries continuity method to prove existence} of a constant Chern scalar curvature metric in $\{\omega\}$ of class $\mathcal{C}^{2,\alpha}$, and then to obtain smooth solutions by regularity.

We set up the following {\bfseries continuity path}.
We fix $\lambda=\Gamma_X(\{\omega\})$, and fix the above metric $\omega$ with $S^{Ch}(\omega)<0$ as reference metric in the conformal class.
Consider the equation, varying $t\in[0,1]$:
$$ \mathrm{ChYa}(t,f) := \Delta^{Ch} f + t S^{Ch} (\omega)  - \lambda \exp(2f/n) + \lambda (1-t) = 0\; . $$
For $\alpha \in (0,\,1)$, we have the map
$$ \mathrm{ChYa} \colon [0,\,1]\times\mathcal{C}^{2,\alpha}(X ; \mathbb{R}) \rightarrow  C^{0,\alpha} (X ; \mathbb{R}) \;. $$
Let us define
$$ T:= \left\{ t\in [0,\, 1] \;\middle|\;  \exists f_t \in C^{2 , \, \alpha} (X ; \mathbb{R}) \mbox{ such that } \mathrm{ChYa}(t,\, f_t)=0 \right\} \; . $$
Now, $0\in T$ since $\mathrm{ChYa}(0,0)=0$ as we can directly check. Thus, {\bfseries $T$ is non empty}.

\medskip

In order to show that {\bfseries $T$ is open} we argue as follows. The implicit function theorem guarantees that $T$ is open
as long as the linearization of $\mathrm{ChYa}$ with respect to the second variable is bijective. That is, we claim that, {\em for a fixed $t_0 \in T$, with corresponding solution $f_{t_0}$,
the operator
$$ D \colon C^{2 , \, \alpha} (X ; \mathbb{R}) \rightarrow C^{0 , \, \alpha} (X ; \mathbb{R}) $$
defined by
$$ v \mapsto Dv : = \Delta^{Ch} v - \lambda \exp\left(2f_{t_0}/n\right)\cdot 2v/n  $$
is bijective}.
We begin by showing the injectivity of the linear operator $D$.
Since $D$ is an elliptic operator, by standard regularity theory, we can use the ordinary maximum principle for functions in the kernel: if $v$ belongs to $\ker D$, then at a maximum point $p$ for $v$ there holds
\begin{align}
 - \lambda \exp(2f_{t_0}(p)/n)\cdot 2v(p)/n \leq 0 \, ,
\end{align}
and whence $v(p)\leq 0$. Similarly, at a minimum point $q$ for $v$, there holds $v(q)\geq 0$.
Whence, $\ker D = \{0\}$. 

To see that {$D$ is surjective}, we notice that, since the index of the linear operator $D$ is the same as the index of the Laplacian, (from whom $D$ differs by a compact operator,)
the injectivity directly implies surjectivity.

\medskip

Finally, we claim that {\bfseries $T$ is closed}. We first prove {uniform $L^\infty$ estimates of the solution}.
\begin{lem}\label{lem: C zero unif estimate}
 If $\{ t_n\} \subset T$ and $f_{t_n} \in\mathcal{C}^{2,\alpha}(X;\R)$ are such that $\mathrm{ChYa} (t_n , \, f_{t_n} ) = 0$ for any $n$, then there exists a positive constant $C_0$, depending only on $X$, $\omega$, $\lambda$,
 such that, for any $n$,
\begin{align}
 \|f_{t_n}\|_{L^\infty} \leq C_0 \; .
\end{align}
\end{lem}
\begin{proof}
 Consider the following equality
\begin{align}
 \Delta^{Ch} f_{t_n} + t_n S^{Ch} (\omega)  - \lambda \exp\left(2f_{t_n}/n\right) + \lambda (1-t_n) = 0 \, ,
\end{align}
and suppose that $p$ is a maximum point for $f_{t_n}$. Then, at $p$, there holds
\begin{align}
 - \lambda \exp\left(2f_{t_n}(p)/n\right) \leq -t_n S^{Ch} (\omega)(p) - \lambda (1-t_n) \leq -\left( \max_X S^{Ch} (\omega) \right) -\lambda \, ,
\end{align}
(recall that $\lambda < 0$ and that $S^{Ch} (\omega)$ is a negative function).
Similarly, at a minimum point for $f_{t_n}$, say $q$, there holds
\begin{eqnarray*}
  - \lambda \exp\left(2f_{t_n}(q)/n\right) &\geq& -t_n S^{Ch} (\omega)(q) - \lambda (1-t_n) \geq t_n (-S^{Ch}(\omega)(q) +\lambda) -\lambda\\[5pt]
  &\geq& \min \left\{ \min_X \left(-S^{Ch}(\omega)\right) ,\, -\lambda  \right\} >0\; .
\end{eqnarray*}
The above estimates provide the wanted uniform constant $C_0$.  
\end{proof}
Now, consider the equality
$$ L_n f_{t_n} = \lambda \exp\left(2f_{t_n}/n\right) \, ,$$
where
$$ L_n f : = \Delta^{Ch} f + t_n S^{Ch} (\omega) + \lambda (1-t_n) $$
is regarded as an elliptic operator. The estimate of Lemma \ref{lem: C zero unif estimate} gives a uniform $L^\infty$ control of the right-hand side $\lambda \exp\left(2f_{t_n}/n\right)$ of the equation. Then, by iterating the Calderon-Zygmund inequality and using Sobolev embeddings, we find, let us say, 
an a-priori $\mathcal{C}^{3}$ uniform bound. Thus by the Ascoli-Arzel\`a theorem, we can take a converging subsequence to a solution of the equation, and conclude that {\em the set of parameters $t$ for which the equation $\mathrm{ChYa} (t,\, \sspace )$ admits a $\mathcal{C}^{2,\alpha}$ solution is closed in the Euclidean topology}.

\medskip

So far we accomplished the existence of a $\mathcal{C}^{2,\alpha}$ solution $f$ to the Chern-Yamabe equation, $\mathrm{ChYa}(1,f)=0$. 
We are left to prove the {\bfseries smooth regularity of the solution}. But this follows by the usual bootstrap argument via Schauder's estimates. 

\medskip

Now we turn to the {\bfseries uniqueness} issue.
Suppose that $\omega_1=\exp(2f_1/n)\omega\in\{\omega\}$ and $\omega_2=\exp(2f_2/n)\omega \in \{\omega\}$ have constant Chern scalar curvature equal to $\lambda_1<0$ and $\lambda_2<0$ respectively.
Hence we have the equations
$$
\Delta^{Ch}_{\omega} f_1 + S^{Ch}(\omega) = \lambda_1 \exp \left( 2f_1/n \right)
\qquad \text{ and } \qquad
\Delta^{Ch}_{\omega} f_2 + S^{Ch}(\omega) = \lambda_2 \exp \left( 2f_2/n \right) \;.
$$
Then we get the equation
$ \Delta^{Ch}_{\omega} (f_1-f_2) = \lambda_1 \exp \left( 2f_1/n \right) - \lambda_2 \exp \left( 2f_2/n \right)$.
At a maximum point $p$ for $f_1-f_2$, we find
$f_1 (p) - f_2 (p) \;\leq\; n\log\left( \lambda_2/\lambda_1\right)/2$.
Similarly, at a minimum point $q$, we have
$f_1 (q) - f_2 (q) \;\geq\; n\log\left( \lambda_2/\lambda_1\right)/2$.
Thus
$$ f_1 = f_2 + n\log\left( \lambda_2/\lambda_1\right)/2 \;.$$
By considering now our normalization condition, we get the uniqueness.
The last statements follow now from Proposition \ref{prop:gamma-sign} and the assumption, and from Lemma \ref{lem:CGG} and the uniqueness.
\end{proof}

By combining Theorems \ref{thm:cy-zero-gaud} and \ref{thm:cy-neg-gaud} with Proposition \ref{prop:kod-gauddegree}, we obtain the following corollary.

\begin{cor}\label{cor:kod-non-negative}
 Let $X$ be a compact complex manifold with non-negative Kodaira dimension. Then, for any $\{\omega\}\in\mathcal{C}^H_X$, it holds $\widetilde{\mathcal{C}h\mathcal{Y}a(X,\{\omega\})} =\{p\}$.
\end{cor}

\section{Towards the case of positive Gauduchon degree}\label{sec:pos-gaud}

In this last section we give some remarks on the behaviour of the Chern-Yamabe problem in the case of \emph{positive} Gauduchon degree. As we mentioned in the Introduction, it is evident that many of the good analytical properties of the Chern-Yamabe equation disappear for $\Gamma_X(\{\omega\})>0$ (mostly the ones related to the maximum principle). In principle, it may be as well that the ``naive'' Chern-Yamabe conjecture \ref{conj:cy} does not hold for every conformal class. In any case, it seems interesting to find some sufficient and, possibly, necessary conditions which ensure the existence of metrics with positive constant Chern scalar curvature.

After showing some easy examples which show that (non-K\"ahler) metrics with constant positive Chern scalar exist, we discuss a couple of analytic observations that may be useful for the general study of the problem. Finally we give a ``sufficient'' criterion for the existence of positive constant Chern scalar curvature metrics (based on deformations of flat Chern-scalar metric).

\subsection{Elementary examples}

As we discussed, conformal classes having positive Gauduchon degrees could exist only on manifolds with Kodaira dimension equal to $-\infty$. If it happens that there exists at least one non-trivial antiplurigenera, then \emph{all} the Gauduchon degrees are positive. One of the easiest (non-K\"ahler) manifold having such property is the standard Hopf surface, and indeed in this case one can easily find a metric with constant positive Chern scalar curvature.
\begin{exa}[Hopf surface]\label{ex:hopf}
 Consider the Hopf surface $X=\left. \C^2\setminus\{0\} \middle\slash G \right.$ where $G=\left\langle z\mapsto z/2 \right\rangle$. It is diffeomorphic to $\mathbb{S}^1\times\mathbb{S}^3$ and it has Kodaira dimension $\mathrm{Kod}(X)=-\infty$. The standard Hermitian metric $\omega(z) = \frac{1}{|z|^2} \omega_0$, where $\omega_0$ is the flat Hermitian metric on $\C^2$, has constant positive Chern scalar curvature. 
\end{exa}

The above example can be easily generalized in other similar situations.

By blowing-up an Hopf surface one obtains examples where the Gauduchon degrees can take either positive, or zero, or negative values \cite{teleman}. We remark that it is not true that every manifold having negative Kodaira dimension admits conformal classes with positive Gauduchon degree, as the following example shows.

\begin{exa}[Inoue surfaces]\label{ex:inoue}
 Let us consider an Inoue surface $X$. It is a compact complex surface in class VII, hence it has $\mathrm{Kod}(X)=-\infty$ and $b_1=1$. In \cite[Remark 4.2]{teleman} it is shown that $\Gamma_X(\{\omega\})<0$ for any $\{\omega\}\in\mathcal{C}^H_X$. 
 
 In particular, combining the above with Theorem \ref{thm:cy-neg-gaud}, we obtain examples of manifolds with negative Kodaira dimension where all conformal classes admit metrics of (negative) constant Chern scalar curvature.
\end{exa}

In by-hand-constructed examples, (which are typically very symmetric,) it usually happens that the metrics are of Gauduchon type. It is then interesting to see that non-Gauduchon metrics with positive constant Chern scalar curvature actually exist. For an easy construction of such manifolds, it is sufficient to take the product of $\C\P^1$ with ``the'' standard round metric $\omega_{FS}$, scaled by a small factor $\epsilon$, with a \emph{non}-Gauduchon solution $(X,\omega)$ of the Chern-Yamabe problem for negative Gauduchon degree, (for example take any non-cscK K\"ahler metric on a manifold $X$ with non-negative Kodaira dimension and apply our previous Theorems to find $\omega$). Then it is immediate to check that $(\C\P^1\times X, \epsilon \mathrm{pr}_1^\ast(\omega_{FS}) + \mathrm{pr}_2^\ast(\omega))$ has constant \emph{positive} Chern scalar curvature provided $\epsilon$ is small enough, but it is not Gauduchon.

\subsection{Chern-Yamabe equation as an Euler-Lagrange equation}

In order to attack the problem when the Gauduchon degree is positive, it may be useful to see if our equation 
\begin{equation}\label{eq}\tag{\ref{eq:cy}}
\Delta^{Ch}f + S=\lambda \exp(2f/n) \;,
\end{equation}
where $S:=S^{Ch}(\omega)$, can be written as an Euler-Lagrange equation of some functional, with multiplier (and then try to apply direct methods to converge to a minimizer). However, as we are now going to show, in general such nice property does not hold. For this, let us say that a Hermitian conformal class $\{\omega\}$ is called \emph{balanced} if its Gauduchon representative $\eta$ is balanced, i.e., $d\eta^{n-1}=0$. Note that, if $\eta$ is balanced, then $\Delta^{Ch}_{\eta}=\Delta_{d,\eta}$ on smooth functions. We denote the set of such conformal classes as $\mathcal{C}_X^b \subseteq \mathcal{C}_X^H$.

\begin{prop}\label{prop:variational-form}
Let $X$ be a compact complex manifold of complex dimension $n$ and let $\{\omega\}\in\mathcal{C}^H_X$.
The $1$-form on $\mathcal{C}^\infty(X; \R)$
$$\alpha\colon h\longmapsto\int_X h\cdot(df,\theta)_\omega d\mu_\omega$$
is never closed, beside the case when it is identically zero, which happens only if $\omega$ is balanced. 

It follows that equation \eqref{eq} can be seen as an Euler-Lagrange equation (with multiplier) for the standard $L^2$ pairing if and only if $\{\omega\} = \{\eta\} \in \mathcal{C}_X^b$ with $\eta$ balanced. In this case, the functional takes the form
\begin{equation}\label{eq:el-functional}
\mathcal{F}(f) := \frac{1}{2}\int_X |d f|_\eta^2 d\mu_\eta + \int_X S^{Ch}(\eta) f d\mu_\eta \;,
\end{equation}
subject to the integral constraint

\begin{equation}\label{eq:el-constraint}
\int_X \exp(2f/n)d\mu_\eta=1\;.
\end{equation}
\end{prop}

Note that the Lagrange multiplier coincide with the Gauduchon degree of the balanced conformal class.

\begin{proof}
Note that $d\alpha_f(h,g)\equiv 0$ if and only if $\int_X h(dg,\theta) d\mu_\omega=\int_{X} g (dh,\theta)d\mu_\omega$, for all $h,g \in C^\infty(X;\R)$.
In particular, by taking $h$ to be a constant and $g$ arbitrary, we have, by integration by part, that $\theta$ must be co-closed (i.e., $\omega$ must be Gauduchon).
Now, since the integral $\int_X (d(gh),\theta)d\mu_\omega$ is always zero for a Gauduchon metric $\omega$, it follows that $\int_{X} g (dh,\theta)\omega^n=0$ for all $g,h$. Thus $\theta = 0$ as claimed.
\end{proof}

\begin{rmk}
We notice here some facts. 
\begin{itemize}
 \item The above proposition makes sense even in the case of negative Gauduchon degree. (As a natural question one may as whether the constant Chern-Scalar metric is an absolute minimum in this case.)
 \item Thanks to Lemma \ref{lem:CGG}, the functional $\mathcal{F}$ is ``gauge'' invariant, that is, $\mathcal{F}(f\circ \phi)=\mathcal{F}(f)$ for any $\phi\in\mathcal{H}\mathrm{Conf}(X,\{\omega\})$.
 \item Instead of working with constraints, one can instead use the following scale invariant normalization:
 $$\mathcal{F}^\ast(f):=\frac{\mathcal{F}(f)}{n} -\lambda \log \left( \int_X \exp(2f/n)d\mu_\eta \right).$$
 Its Euler-Lagrange equation is now given by:
 $$ \Delta^{Ch}f + S = \lambda \frac{\exp(2f/n)}{\left( \int_X \exp(2f/n)d\mu_\eta \right)},$$
 which essentially coincides with the constant Chern scalar curvature equation, since we can always normalize $\left( \int_X \exp(2f/n)d\mu_\eta \right)=1$. We remark that some care is needed when trying to define the functional $\mathcal{F}^\ast$ on some Sobolev spaces due to the exponential term involved in its expression.
 \item At a smooth critical point $f$ of $\mathcal{F}$ for $\lambda>0$, one gets (by applying Jensen's inequality): $$\mathcal{F}(f) \geq -\frac{1}{2}\int_X |d f|_\eta^2 d\mu_\eta ;$$ 
however, it is not clear if the functional is bounded below.
\end{itemize}

 \end{rmk}

\subsection{On Chern-Yamabe flow}

Another possible candidate to study the existence problem in the case of positive Gauduchon degree (but which makes of course sense in general) can be the \emph{Chern-Yamabe-flow}.

\begin{defi}\label{CYF}
Let $X$ be a compact complex manifold and let $\{\omega\}\in\mathcal{C}^H_X$ be a Hermitian conformal class.
Fix any $\omega \in \{\omega\}_1$ (e.g., the Gauduchon representative with unit volume) as reference metric in the class. We call \emph{Chern-Yamabe flow} the following parabolic equation:
\begin{equation}\tag{ChYa-f}
\frac{\partial f}{\partial t}=-\Delta_\omega^{Ch}f -S^{Ch}(\omega) +\lambda \exp{\left(\frac{2f}{n}\right)}\;.
\end{equation}
\end{defi}
Since the principal symbol of the Chern Laplacian coincides with the one of the standard Laplacian, it follows that the above flow always starts, given any smooth initial datum.

Few remarks are in place. First note that the constraint condition \eqref{eq:cy-constraint} is not preserved under the flow. Nevertheless, it seems interesting to note that, if we take as base point for the flow a \emph{balanced} Gauduchon representative $\eta$ and the initial value being equal to $f_0=0$, then, by a simple calculation, it follows that $\mathcal{F}(f_t)\leq 0$ for $t$ small enough. 

Finally, from a more analytic prospective, it may be interesting to see whether in the case of negative Gauduchon degrees the above Chern-Yamabe flow actually converges to the unique solution we constructed in the previous sections.

\subsection{On weak versions of the Chern-Yamabe problem}
On any compact complex manifold $X$ with some non-vanishing anti-plurigenera, the possible values of Gauduchon degree are always positive. In particular, at the current state of art, we do not even know whether on such manifolds a metric with (necessary) positive constant Chern curvature can actually exists for some conformal class (a kind of weak Chern-Yamabe conjecture). We hope to discuss in a subsequent paper sufficient "computable" conditions on $(X,\{\omega\})$ for the existence of constant Chern scalar curvature metrics. For the purpose of this note, we restrict ourself to show three elementary results along these lines. 

\medskip

The first result is about complex manifolds which admit both conformal classes with both zero and conformal classes with positive Gauduchon degrees.

\begin{prop}\label{PosDef} Let $X$ be a compact complex manifold and let $\eta_{s\in[0,1)}$ be a smooth path of Gauduchon metrics on $X$ such that the Gauduchon degree $\Gamma_X(\{\eta_s\})$ is equal to zero for $s=0$ and is positive for  $s>0$. Then constant Chern scalar curvature metrics, with positive small Chern scalar, exist in all conformal classes $\{\eta_s\}$ for  $s>0$ small enough.
\end{prop}

\begin{proof}
 Denote the smooth function $\Gamma_X(\{\eta_s\})$ with $\Gamma_s$. The equation we want to solve is the usual $\Delta_{\eta_s}^{Ch}f+S^{Ch}(\eta_s)=\Gamma_s \exp(2f/n)$, for $s$ small enough. Then the statement follows immediately by using Theorem \ref{thm:cy-zero-gaud} and the Implicit Function Theorem for Banach manifolds after imposing the natural constraints (compare later in this section for more details). Note that $S^{Ch}\left(\exp(2f_s/n)\eta_s\right)=\Gamma_s>0$ for $s>0$.
\end{proof}

\medskip

The second result, which is a consequence and an example of the above proposition, says that given any compact complex manifold $X$, (eventually having all Gauduchon degrees positive,) it is sufficient to take the product $X\times\Sigma_g$ with some complex curve $\Sigma_g$ of genus $g \geq 2$ to have some conformal classes in the product admitting positive constant Chern scalar curvature metrics.

\begin{prop}\label{prop:pos-product-curve}
Let $X$ be a compact complex manifold of complex dimension $\dim_\C X \geq2$ endowed with a Hermitian conformal class $\{\omega\}$ with $\Gamma_X(\{\omega\})>0$. Then $X\times \Sigma_g$, with $\Sigma_g$ compact complex curve of genus $g\geq2$, admits  Hermitian conformal classes with metrics having positive constant Chern scalar curvature.
\end{prop}

The proposition is a consequence of the following lemma.

\begin{lem} Let $X$ be a compact complex manifold of complex dimension $n\geq2$ endowed with a Gauduchon metric $\eta$ with $\Gamma_X(\{\eta\})>0$ and $\int_X \eta^n=1$. Let $\Sigma_g$ be a compact complex curve of genus $g\geq2$ equipped with a (K\"ahler) metric $\omega$ with $\int_{\Sigma_g} \omega=1$. Define the family of Hermitian products $(X\times \Sigma_g, \omega_\delta)$, where $\omega_\delta:=\mathrm{pr}_1^\ast(\eta)+\delta \mathrm{pr}_2^\ast(\omega)$. Then:
\begin{itemize}
 \item $\omega_\delta$ is a smooth path of Gauduchon metrics;
 \item the Gauduchon degree $\Gamma_\delta:=\Gamma(\{\omega_\delta\})$ is a smooth function satisfying $$\mathrm{sgn}\left(\Gamma_\delta\right)=\mathrm{sgn}\left(\Gamma_X(\{\eta\})\delta+4\pi(1-g)\right),$$
 where $\mathrm{sgn}$ denotes the sign function.
\end{itemize}
Thus, by Proposition \ref{PosDef}, it exists $\epsilon=\epsilon(X\times\Sigma_g,\eta,\omega)>0$ such that for all $\delta \in ( 4\pi(g-1)\Gamma_X(\{\eta\})^{-1},4\pi(g-1)\Gamma_X(\{\eta\})^{-1}+\epsilon)$  the conformal classes $\{\omega_\delta\}$ admit metrics with positive constant Chern scalar curvature.
\end{lem}

\begin{proof}
 Abusing of notation, let us write $\omega_\delta=\eta+\delta \omega$. It follows by dimension reasons that $\omega_\delta^n=\eta^n+\delta n \eta^{n-1}\wedge\omega$. Thus $i\partial \bar\partial \omega_\delta^n=0$, i.e., $\omega_\delta$ is a smooth family of Gauduchon metrics. Observe that $\omega_\delta$ is not normalized, in particular  the quantity $\int_{X\times\Sigma_g} S^{Ch}(\omega_\delta)\omega_{\delta}^{n+1}$  is only proportional, (by an explicitly computable smooth positive function of $\delta$,) but not equal to the Gauduchon degree. Thus
 $$ \mathrm{sgn}\left( \Gamma_\delta\right) \;=\; \mathrm{sgn}\left(\int_{X\times\Sigma_g}(S^{Ch}(\eta)+\delta^{-1}S(\omega))\delta(n+1)\eta^n\wedge\omega\right) \;=\; \mathrm{sgn}\left(\Gamma_X(\{\eta\})\delta+4\pi(1-g)\right) \;,$$
 where the last equality simply follows by Fubini and Gauss-Bonnet theorems. Since, for $\delta=4\pi(g-1)\Gamma_X(\{\eta\})^{-1}$, the Gauduchon degree $\Gamma_\delta$ vanishes, we can apply Proposition \ref{PosDef} to conclude the existence of positive constant Chern scalar curvature metrics for $\delta$ slightly bigger than the zero threshold.
\end{proof}

\medskip

Finally, the third result is another application of the Implicit Function Theorem, when Chern scalar curvature is "small"; this is in view of the case of small (possibly positive) Gauduchon degree.

\begin{thm}\label{thm:small-lambda}
 Let $X$ be a compact complex manifold of complex dimension $n$, and let $\{\omega\} \in \mathcal{C}^H_X$. Let $\eta\in\{\omega\}$ be the unique Gauduchon representative of volume $1$. There exists $\varepsilon>0$, depending just on $X$ and $\eta$, such that, if $\left\|S^{Ch}(\eta)\right\|_{\mathcal{C}^{0,\alpha}(X)}<\epsilon$, for some $\alpha\in(0,1)$, (in particular, this yields $\Gamma_X(\{\eta\})<\varepsilon$,) then there exists a constant Chern scalar curvature metric in $\{\eta\}_1$.
\end{thm}

\begin{proof}
 Fix $\alpha\in(0,1)$. Consider the Banach manifolds of class $\mathcal{C}^1$
 \begin{eqnarray*}
 {\mathcal X} &:=& \left\{ (u,\lambda,S)\in \mathcal{C}^{2,\alpha}(X;\R) \times \R \times \mathcal{C}^{0,\alpha}(X;\R) \;\middle\vert\; \int_X S d\mu_\eta - \lambda \cdot \int_X \exp(2u/n) d\mu_\eta = 0 \right\} \;, \\[5pt]
 {\mathcal Y} &:=& \left\{ w \in \mathcal{C}^{0,\alpha}(X;\R) \;\middle\vert\; \int_X w d\mu_{\eta} = 0 \right\} \;.
 \end{eqnarray*}
 Consider the map of class $\mathcal{C}^1$ given by
 $$ F \colon {\mathcal X} \to {\mathcal Y} \;, \qquad F(u,\lambda,S) \;=\; \Delta^{Ch}_\eta u + S - \lambda\cdot \exp(2u/n) \;. $$
 Note that $F(0,0,0)=0$, and that
 $$ \left. \frac{\partial F}{\partial u} \right\lfloor_{(0,0,0)} \;=\; \Delta^{Ch}_{\eta} $$
 is invertible on $T_0\mathcal{Y}\simeq\mathcal{Y}$.
 
 Hence, by applying the implicit function theorem, there exists a neighbourhood $U$ of $(0,0)$ in $\R \times \mathcal{C}^{0,\alpha}(X;\R)$ and a $\mathcal{C}^1$ function $\left(\tilde u, \mathrm{id}\right) \colon U \to {\mathcal X}$ such that
 $$ F\left(\tilde u(\lambda,S),\lambda,S\right)\;=\;0 \;. $$
 
 Take $\varepsilon>0$, depending on $X$ and $\eta$, such that
 $$ \left\{ \lambda \in \R \;\middle\vert\; |\lambda|<\varepsilon \right\} \times \left\{ S \in \mathcal{C}^\infty(X;\R) \;\middle\vert\; \left\| S \right\|_{\mathcal{C}^{0,\alpha}(X;\R)} < \varepsilon \right\} \;\subset\; U \;. $$
 
 Let $\eta$ be such that $S^{Ch}(\eta)\in\mathcal{C}^\infty(X;\R)$ satisfies $\left\|S^{Ch}(\eta)\right\|_{\mathcal{C}^{0,\alpha}(X;\R)}<\varepsilon$. By taking $\lambda:=\Gamma_X(\{\eta\})$, (as expected by Proposition \ref{prop:gamma-sign},) we have also
 $$ \lambda \;=\; \Gamma_X(\{\eta\}) \;=\; \int_X S^{Ch}(\eta) d\mu_\eta \;\leq\; \left\| S^{Ch}(\eta) \right\|_{L^\infty(X)} \cdot \int_X d\mu_\eta \;<\; \varepsilon \;. $$
 Since $\left(\lambda, S^{Ch}(\eta)\right)\in U$, take $u:=\tilde u\left(\lambda,S^{Ch}(\eta)\right)\in\mathcal{C}^{2,\alpha}(X;\R)$ as above. By regularity, $u$ belongs in fact to $\mathcal{C}^\infty(X;\R)$. Then $\exp(2u/n)\eta \in \{\eta\}_1$ has constant Chern scalar curvature equal to $\lambda$.
\end{proof}

\begin{rmk}
In the previous statement, note the condition on the bound of $S^{Ch}(\eta)$ in terms of $\varepsilon$, where $\varepsilon$ depends by $\eta$ itself. Notwithstanding, we notice that the previous result is not empty: there are examples to which Theorem \ref{thm:small-lambda} can be applied. For example, take a compact complex manifold $X$ endowed with a Gauduchon metric $\eta$ having $S^{Ch}(\eta)=0$. Let $\varepsilon_\eta>0$ be the constant in Theorem \ref{thm:small-lambda}: it depends continuously on $\eta$. Consider a perturbation $\omega$ of $\eta$ in $\mathcal{C}^{2,\alpha}(X;T^*X\otimes T^*X)$, where $\alpha\in(0,1)$, such that $\|S^{Ch}(\omega)-S^{Ch}(\eta)\|_{\mathcal{C}^{2,\alpha}(X;\R)}<\varepsilon_\eta/4$ with $\Gamma_X(\{\omega\})>0$, 
and $|\varepsilon_\omega - \varepsilon_\eta |<\varepsilon_\eta/4$, where $\varepsilon_\omega$ is the constant in Theorem \ref{thm:small-lambda} associated to $\omega$. Then in particular $\|S^{Ch}(\omega)\|_{\mathcal{C}^{2,\alpha}(X;\R)}<\varepsilon_\omega$, and Theorem \ref{thm:small-lambda} applies, yielding a constant positive Chern scalar curvature metric in the conformal class of $\omega$.
\end{rmk}

\subsection{Non-uniqueness of solutions for positive Gauduchon degree}
In this section, we show that uniqueness does not hold, in general, in the case of positive Gauduchon degree.

\begin{prop}\label{prop:non-uniq-positive}
 In general, on a compact complex manifold $X$, constant Chern scalar curvature metrics in a fixed conformal class $\{\omega\}$ with positive Gauduchon degree are not unique, even up to the action of $\mathcal{G}_X(\{\omega\})$.
\end{prop}

\begin{proof}
 In \cite{dLPZ}, local rigidity and multiplicity of solutions for the classical Yamabe problem on, in particular, products of compact Riemannian manifolds are investigated. Notwithstanding, it is not clear to us whether the argument in \cite[Theorem 2.1]{smoller-wasserman}, on which \cite[Theorem A.2]{dLPZ} is founded, still holds true in our non-variational setting. Hence, we apply similar arguments by using instead a version of the Krasnosel'skii Bifurcation Theorem, see, e.g., \cite[Theorem II.3.2]{kielhofer}, proving bifurcation in case of odd crossing numbers. More precisely, we apply \cite[Theorem A]{westreich} by D. Westreich.

 \medskip
 
 Consider the compact complex manifold $X:=\C\P^1\times\C\P^1\times\C\P^1$ endowed with the path $\left\{g_\lambda\right\}_{\lambda\in(0,+\infty)}$ of Hermitian metrics, where
 $$ \omega_{\lambda} \;:=\; V(\lambda)^{-1/3} \cdot \left( \pi_1^*\omega_{FS} + \pi_2^*\omega_{FS} + \lambda \cdot \pi_3^*\omega_{FS} \right) \;; $$
 here, $\pi_j\colon X \to \C\P^1$ denotes the natural projection, and $\omega_{FS}$ is the Fubini and Study metric on $\C\P^1$; the constant $V(\lambda)=(4\pi)^3\lambda$ is chosen in such a way that $\mathrm{Vol}(\omega_\lambda)=1$.
 
 We consider the Banach spaces
 $$ \mathcal{D} := \left\{ f \in \mathcal{C}^{2,\alpha}(X;\R) \;\middle\vert\; \int_X f d\mu_{\omega_1}=0 \right\} \;, $$
 and
 $$ \mathcal{E} := \left\{ \varphi \in \mathcal{C}^{0,\alpha}(X;\R) \;\middle\vert\; \int_X \varphi d\mu_{\omega_1}=0 \right\} \;, $$
 (note that $d\mu_{\omega_\lambda}=\lambda V(\lambda)^{-1} d\mu_{\omega_1}$ for any $\lambda\in(0,+\infty)$,)
 and the smooth map
 $$ F \colon \mathcal{D} \times (0,+\infty) \to \mathcal{E} \;, \qquad F\left( f, \lambda \right) \;:=\; S^{Ch}\left(\exp(2f/3)\omega_\lambda\right) - \int_X S^{Ch}\left(\exp(2f/3)\omega_\lambda\right) d\mu_{\omega_\lambda} \;. $$
 Note that $\exp(2f/3)\omega_\lambda \in \{\omega_\lambda\}$ is constant Chern scalar curvature if and only if $F\left(f, \lambda\right)=0$.
 
 For any $\lambda\in(0,+\infty)$, we have that $F(0,\lambda)=0$, since $\omega_\lambda$ is cscK, with constant curvature $S(\lambda):=2 \cdot (2+\lambda^{-1}) \cdot V(\lambda)^{1/3}$. This gives condition {\itshape (a)} of \cite[Theorem A]{westreich}. Conditions {\itshape (b)} and {\itshape (c)} of \cite[Theorem A]{westreich} are guaranteed as in \cite[Remark at page 610]{westreich}.
 
 Consider
 $$ A(\lambda)\colon\mathcal{D}\to\mathcal{E}\; \qquad A(\lambda)[\beta] \;:=\; \left.\frac{\partial}{\partial f} F\right\lfloor_{(0,\lambda)} [\beta] \;=\; -\frac{2}{3}S(\lambda)[\beta]+\Delta_{\omega_\lambda}[\beta] \;, $$
 which is a closed Fredholm operator of index $0$, for any $\lambda\in(0,+\infty)$. In particular, we study the value
 $$ \lambda_0 \;:=\; 1/4 \;. $$
 
 For condition {\itshape (d)} of \cite[Theorem A]{westreich}, we prove that $\left.\frac{\partial}{\partial f} F\right\lfloor_{(0,\lambda_0)}$ is a Fredholm operator of index $0$ such that $\dim \ker \left.\frac{\partial}{\partial f} F\right\lfloor_{(0,\lambda_0)}$ is odd.  Indeed, recall that the eigenvalues of $\Delta_{\omega_{FS}}$ on $\C\P^1$ are $j(j+1)$ varying $j\in\N$, with multiplicity $2j+1$ respectively. We have hence that $\ker \left.\frac{\partial}{\partial f} F\right\lfloor_{(0,\lambda_0)}$ is generated by
 $$ V^{(2,1,0)} \cup V^{(1,2,0)} \cup V^{(0,0,1)} \;, $$
 where, meaning here below for $\beta$ a generic element of $\mathcal{C}^\infty(\C\P^1\times\C\P^1\times\C\P^1;\R)$,
 $$V^{(j_1,j_2,j_3)}:= \left\{\beta (p_1,\, p_2,\, p_3 ) := \beta_1 (p_1) \beta_2 (p_2) \beta_3 (p_3) \;\middle\vert\; \Delta^{(\ell)}\beta_\ell=j_\ell(j_\ell+1)\beta_\ell \text{ for } \ell\in\{1,2,3\}\right\},$$ and, with abuse of notation, $\Delta^{(\ell)}:=\pi_\ell^*\circ\Delta_{\omega_{FS}}\circ\iota_\ell^*$ denotes the Laplacian on the $\ell$-th factor $\C\P^1$, where $\pi_\ell\colon X \to \C\P^1$ and $\iota_\ell\colon \C\P^1\to X$ are the natural projection, respectively inclusion, of the $\ell$-th object $\C\P^1$ in $X$, for $\ell\in\{1,2,3\}$.  In fact, we have
 $$
 \dim \ker \left.\frac{\partial}{\partial f} F\right\lfloor_{(0,\lambda_0)}
 \;=\; \dim V^{(2,1,0)}+\dim V^{(1,2,0)}+\dim V^{(0,0,1)}
 \;=\; 15 + 15 + 3
 \;=\; 33 \;,
 $$
 which is odd.

 Compute
 $$ \left.\frac{\partial^2}{\partial f \partial \lambda} F\right\lfloor_{(0,\lambda_0)} [\beta] \;=\; -\frac{2}{3}S'(\lambda_0)[\beta] + \left.\frac{d}{d\lambda}\right\lfloor_{\lambda_0}\Delta_{\omega_\lambda}[\beta] \;. $$
 We now verify condition {\itshape (e)} of \cite[Theorem A]{westreich}.
 More precisely, we have to show that, if $\tilde v\in \ker \left.\frac{\partial}{\partial f} F\right\lfloor_{(0,\lambda_0)} \setminus \{0\}$, then $\left.\frac{\partial^2}{\partial f \partial \lambda} F\right\lfloor_{(0,\lambda_0)} [\tilde v] \not \in \mathrm{imm} \left.\frac{\partial}{\partial f} F\right\lfloor_{(0,\lambda_0)}$. 
 With the above notation, we have
 $$
 \left.\frac{\partial^2}{\partial f \partial \lambda} F\right\lfloor_{(0,\lambda_0)} [\beta] \;=\;
 \frac{4\pi}{3}\cdot 4^{2/3} \cdot \left( 8[\beta] + \Delta^{(1)}[\beta] + \Delta^{(2)}[\beta] - 8 \Delta^{(3)}[\beta] \right)
 \;.
 $$
 If $\tilde v \in  V^{(j_1,j_2,j_3)}$, with $(j_1,j_2,j_3)\in \{(2,1,0)\,,\,(1,2,0)\,,\,(0,0,1)\}$, then $\left.\frac{\partial^2}{\partial f \partial \lambda} F\right\lfloor_{(0,\lambda_0)} [\tilde v]$ is a non-zero constant multiple of $\tilde v$ itself, where the constants depend on $\lambda_0$ and $(j_1,j_2,j_3)$. In any case, $\tilde v \in \ker \left.\frac{\partial}{\partial f} F\right\lfloor_{(0,\lambda_0)} \perp \mathrm{imm} \left.\frac{\partial}{\partial f} F\right\lfloor_{(0,\lambda_0)}$, the operator $\left.\frac{\partial}{\partial f} F\right\lfloor_{(0,\lambda_0)}$ being self-adjoint. This proves the condition above.
  
 Hence, by \cite[Theorem A]{westreich}, we get that $(0,\lambda_0)$ is a bifurcation instant for the equation $F(f,\lambda)=0$. In particular, we get a non-trivial solution $\exp(2f/3)\omega_\lambda\in\{\omega_\lambda\}$ for some $\lambda$ near $\lambda_0$, whence non-uniqueness.

 Notice that the function $f \in \mathcal{D}$, being a solution of \eqref{eq:cy}, is actually a smooth function; moreover the metrics $\exp(2f/3)\omega_\lambda$ and $\omega_\lambda$ in $\{\omega_\lambda\}$ are not equivalent under the action of $\mathcal{H}\mathrm{Conf}(X, \{\omega_{\lambda}\}) \times \mathbb{R}^+$, since 
 $\mathcal{H}\mathrm{Conf}(X, \{\omega_{\lambda}\}) \cong \mathcal{H}\mathrm{Isom} (X, \omega_{\lambda})$ and, from the very definition of $\mathcal{D}$, $\exp(2f/3)\omega_\lambda$ cannot be a non-trivial scaling of $\omega_{\lambda}$.
 \end{proof}

\appendix

\section{Formulas for Chern Laplacian and Chern scalar curvature}\label{app:formulas}

For the sake of completeness, we collect in this section some explicit computations used to derive the expressions for the Chern Laplacian and the Chern scalar curvature. 

\renewcommand{\referenza}{\ref{lem:chern-laplacian}}
\begin{lem*}[{\cite[pages 502--503]{gauduchon-mathann}}]
 Let $X$ be a compact complex manifold of complex dimension $n$ endowed with a Hermitian metric $\omega$ with Lee form $\theta$. The Chern Laplacian on smooth functions $f$ has the form
 $$ \Delta^{Ch}f := \frac{1}{2} ( \omega ,\, dd^c f) = \Delta_d f + (df,\, \theta )_\omega \;. $$
\end{lem*}

\begin{proof}
 Denote by $h$ the Hermitian metric of corresponding $(1,1)$-form $\omega$. Consider the following pairing of $2$-forms:
$$ (\alpha ,\, \beta) = \sum_{A,B,C,D} h^{AB}h^{CD}\alpha_{AD}\beta_{CB} \, , $$
where the sub and superscripts range on a frame of $T^{\mathbb{C}}X$.

Now we fix $\alpha = \omega$ and $\beta = dd^c f$ and we consider two cases.

The first is when the frame $\{Z_j\}_j$ is a complex coordinates one. We claim that in this case
$$ (\omega , \, dd^c f) = -4 \sum_{i,j=1}^{n} h^{i\bar j} \partial_i \partial_{\bar j}f \; . $$
In fact we compute
\begin{eqnarray}\label{equa: ddc in the complex coordinate case}
 (\omega , \, dd^c f) &=& \sum_{i,j,k,\ell=1}^{n} h^{i\bar j}h^{k\bar \ell} h(JZ_i , Z_{\bar \ell})dd^c f (Z_k  , \, Z_{\bar j})\\[5pt]
\nonumber &&+\sum_{i,j,k,\ell=1}^{n} h^{\bar i j}h^{\bar k \ell} h( J Z_{\bar i} , Z_\ell)dd^c f (Z_{\bar k}  , \, Z_j)\\[5pt]
\nonumber &=&\sum_{k,j=1}^{n} \sqrt{-1} h^{k\bar j}dd^c f (Z_k  , \, Z_{\bar j})
+\sum_{j,k=1}^{n} (-\sqrt{-1} )h^{\bar k j} dd^c f (Z_{\bar k}  , \, Z_j)\\[5pt]
\nonumber &=&\sum_{k,j=1}^{n} 4\sqrt{-1}^2 h^{k\bar j} \partial_k \partial_{\bar j}f
\end{eqnarray}

The second case which we consider is that of a Hermitian $g$-orthonormal frame 
\begin{align}\label{equa: ortonormal frame}
 \{v_i\}_{i=1,\ldots, 2n} = \{e_i , Je_i\}_{i=1,\ldots, n}\; .
\end{align}
In this second case we claim that
$$ (\omega , \, dd^c f) = 2 \sum_{i=1}^{n} dd^c f (J e_i ,\, e_i)\;. $$
In fact, by definition
\begin{eqnarray*}
(\omega , \, dd^c f)&=& \sum_{a,b=1}^{2n} \omega (v_a , \, v_b) dd^c f (v_b , v_a)   \\[5pt]
      \nonumber&=& \sum_{a,b=1}^{2n} h (J v_a , \,  v_b ) dd^c f (v_b , v_a) \\[5pt]
     \nonumber&=&  \sum_{a=1}^{2n}  dd^c f (J v_a , v_a) \\[5pt]
     \nonumber&=& 2 \sum_{a=1}^{n}  dd^c f (J e_a , e_a) \;.
\end{eqnarray*}

Now, we are going to prove the formula
\[
 \frac{1}{2} ( \omega ,\, dd^c f) = \Delta_d f + (df , \, \theta)\; .
\]

We begin with claiming that, for a Hermitian $g$-orthonormal frame 
as in \eqref{equa: ortonormal frame}, then there holds
\begin{align}\label{equa: claim one statement}
 \sum_{a=1}^{\dim_\mathbb{C}X}  dd^c f ( e_a , \, J e_a) = -\Delta_d f +  \sum_{i=1}^{2n} \nabla_{v_i}^{LC} v_i \cdot f +
\sum_{j=1}^{n} J[e_j , \, J e_j] \cdot f \, . 
\end{align}
In fact, 
\begin{align}\label{equa: claim one part one}
 dd^c f (e_k , \, J e_k) &= e_k d^c f (Je_k) - J e_k d^c f (e_k) - d^c f [e_k , \, Je_k] \\[5pt]
			 &= e_k d f (-J^2 e_k) - J e_k d f (-Je_k) - d f (-J [e_k , \, Je_k]) \nonumber \\[5pt]
			 &= e_k e_k f + J e_k Je_k  f +J [e_k , \, Je_k] f \, , \nonumber
\end{align}
and 
\begin{eqnarray}\label{equa: claim one part two}
 -\Delta_d f &=& -\delta df \;=\; \sum_{i=1}^{2n} (\nabla_{v_i}^{LC} df)(v_i) \\[5pt]
\nonumber &=&
\sum_{i=1}^{2n} \left(\nabla_{v_i}^{LC} (df(v_i)) -(df)(\nabla_{v_i}^{LC} v_i) \right) \\[5pt]
\nonumber &=& \sum_{i=1}^{2n} \left( v_i v_i f -\nabla_{v_i}^{LC} v_i f \right) \; .
\end{eqnarray}
Whence putting together \eqref{equa: claim one part one} and \eqref{equa: claim one part two} we get \eqref{equa: claim one statement}.
Now we claim that
\begin{align}
 \label{equa: claim two statement}
\sum_{j=1}^{n} J[e_j , \, J e_j] \cdot f = - \sum_{i=1}^{2n} \nabla_{v_i}^{Ch} v_i \cdot f \; .
\end{align}
In fact by definition 
\begin{align}\label{equa: claim two part one}
 \sum_{j=1}^{n} J[e_j , \, J e_j] \cdot f &= \sum_{j=1}^{n} \left( J \nabla_{e_j}^{Ch} J e_j \cdot f 
-J \nabla_{Je_j}^{Ch}  e_j \cdot f -JT^{Ch}(e_j , \, J e_j) \cdot f\right) \\[5pt]
 \nonumber &=  \sum_{j=1}^{n} \left( - \nabla_{e_j}^{Ch}  e_j \cdot f 
- \nabla_{Je_j}^{Ch} J e_j \cdot f \right) \, ,
\end{align}
where we used that $T^{Ch}$ is a $J$-anti-invariant differential form, \cite[page 273]{gauduchon-bumi}, and that $\nabla^{Ch}J=0$. Then \eqref{equa: claim two part one} gives
the claimed \eqref{equa: claim two statement}.
We proved so far that
\begin{align*}
 \frac{1}{2}(\omega , \, dd^c f) = 
\Delta_d f - \sum_{i=1}^{2n} \nabla_{v_i}^{LC} v_i \cdot f + \sum_{i=1}^{2n} \nabla_{v_i}^{Ch} v_i \cdot f \; .
\end{align*}
Whence we will be done when we will show that
\begin{align}\label{equa: claim three statement}
 (df , \, \theta ) = - \sum_{i=1}^{2n} \nabla_{v_i}^{LC} v_i \cdot f + \sum_{i=1}^{2n} \nabla_{v_i}^{Ch} v_i \cdot f \; .
\end{align}
In fact
\begin{align}\label{equa: claim three part one}
 \theta (Y) &:=  \sum_{i=1}^{2n} g(T^{Ch} (Y, \, v_i) ,\, v_i) 
= \sum_{i=1}^{2n} g(\nabla_Y^{Ch}v_i - \nabla_{v_i}^{Ch} Y - [Y, \, v_i]  ,\, v_i)\\[5pt]
\nonumber &= \sum_{i=1}^{2n} \left( \frac{1}{2} Yg(v_i , \, v_i) - v_i g(Y,  \, v_i ) +g(Y,\, \nabla^{Ch}_{v_i} v_i) -g( [Y, \, v_i]  ,\, v_i) \right)\\[5pt]
\nonumber &= \sum_{i=1}^{2n} \left( - g(\nabla_{v_i}^{LC} Y,  \, v_i )- g(Y,  \, \nabla_{v_i}^{LC} v_i ) +g(Y,\, \nabla^{Ch}_{v_i} v_i) -g( [Y, \, v_i]  ,\, v_i) \right)\\[5pt]
\nonumber &= \sum_{i=1}^{2n} \left( - g(\nabla_{Y}^{LC} v_i,  \, v_i )- g(Y,  \, \nabla_{v_i}^{LC} v_i ) +g(Y,\, \nabla^{Ch}_{v_i} v_i)  \right)\\[5pt]
\nonumber &= \sum_{i=1}^{2n} \left( -\frac{1}{2} Y g( v_i,  \, v_i )+ g(Y,  \, \nabla^{Ch}_{v_i} v_i-\nabla_{v_i}^{LC} v_i )  \right)\\[5pt]
\nonumber &= \sum_{i=1}^{2n}  g(Y,  \, \nabla^{Ch}_{v_i} v_i-\nabla_{v_i}^{LC} v_i )  \; .
\end{align}
Now, by taking $Y=\mathrm{grad}_g f$, we see that \eqref{equa: claim three part one} gives precisely \eqref{equa: claim three statement}.
This completes the proof of the Lemma.
\end{proof}

\medskip

The following allows to recover the expression for the Chern scalar curvature, see Proposition \ref{prop:sch-local-coord}.
(We use here, as elsewhere in the note, the Einstein notation.)

\renewcommand{\referenza}{\ref{prop: J parallel connection scalar curvature}}
\begin{prop*}
 Let $X$ be a $2n$-dimensional smooth manifold endowed with an almost complex structure $J$, and let $\nabla$ be an affine connection on $T^{1,0}X$, for which there holds $\nabla J =0$. Take $h$ a Hermitian metric on $X$. Fix $\{ Z_i \}_{i= 1 , \ldots , n}$ a local coordinate frame of vector fields on $T^{1,0}X$.
 Then the scalar curvature of $\nabla$ is given by
 $$ S^{\nabla} = h^{i\bar j} \left( \partial_i \Gamma_{\bar j k}^k - \partial_{\bar j} \Gamma_{ik}^k \right) \;.  $$
\end{prop*}

\begin{proof}
Recall that the scalar curvature is given by
$$ S^{\nabla} = h^{i\bar j}h^{k\bar \ell} R_{i\bar j k \bar \ell}\; , $$
and thus we are led to compute the curvature coefficients 
$$ R_{i\bar j k \bar \ell}= h ( R(Z_i , \, Z_{\bar j} ) Z_k  , \, Z_{\bar \ell} )   \;  .$$
Now, since the frame is coordinate, then all the brackets vanishes.
Moreover, as $\nabla J = 0$, then   $R(Z_i , \, Z_{\bar j} ) Z_k $ is of the same type as $Z_k$.
We have
\begin{align*}
 R_{i\bar j k \bar \ell} &= h\left( \nabla_{Z_i} (\Gamma_{\bar j k}^m Z_m) - \nabla_{Z_{\bar j}} (\Gamma_{ik}^m Z_m) , \, Z_{\bar \ell}  \right)\\[5pt]
 &= \Gamma_{\bar j k}^m \Gamma_{im}^s h_{s \bar \ell} + (\partial_i \Gamma_{\bar j k}^m)h_{m\bar \ell}
    -\Gamma_{ik}^m \Gamma_{\bar j m}^s h_{s\bar \ell} - (\partial_{\bar j}\Gamma_{ik}^m) h_{m\bar \ell} \; .
\end{align*}
Whence, we conclude that
\begin{align*}
 h^{i\bar j}h^{k\bar \ell}R_{i\bar j k \bar \ell} &= h^{i\bar j}h^{k\bar \ell}(\Gamma_{\bar j k}^m \Gamma_{im}^s h_{s \bar \ell} + (\partial_i \Gamma_{\bar j k}^m)h_{m\bar \ell}
    -\Gamma_{ik}^m \Gamma_{\bar j m}^s h_{s\bar \ell} - (\partial_{\bar j}\Gamma_{ik}^m) h_{m\bar \ell}) \\[5pt]
 &=h^{i\bar j} (\Gamma_{\bar j k}^m \Gamma_{im}^k  + \partial_i \Gamma_{\bar j k}^k
    -\Gamma_{ik}^m \Gamma_{\bar j m}^k  - \partial_{\bar j}\Gamma_{ik}^k ) \\[5pt]
 &=h^{i\bar j} (\partial_i \Gamma_{\bar j k}^k - \partial_{\bar j}\Gamma_{ik}^k ) \, ,
\end{align*}
which is the claimed formula.
\end{proof}


\begin{thebibliography}{10}

\bibitem{aubin-book}
T. Aubin, {\em Some nonlinear problems in Riemannian geometry}, Springer Monographs in Mathematics, Springer-Verlag, Berlin, 1998.

\bibitem{dLPZ}
L. L. de Lima, P. Piccione, M. Zedda, On bifurcations of solutions of the Yamabe problem in product manifold, {\em. Ann. Inst. H. Poincaré Anal. Non Linéaire} \textbf{29} (2012), no. 2, 261--277.

\bibitem{delrio-simanca}
H. del Rio, S.~R. Simanca, The Yamabe problem for almost Hermitian manifolds, {\em J. Geom. Anal.} \textbf{13} (2003), no.~1, 185--203.

\bibitem{gauduchon-cras1977}
P. Gauduchon, Le théorème de l'excentricité nulle, {\em C. R. Acad. Sci. Paris Sér. A-B} \textbf{285} (1977), no.~5, A387--A390.

\bibitem{gauduchon-cras1981}
P. Gauduchon, Le théorème de dualité pluri-harmonique, {\em C. R. Acad. Sci. Paris Sér. I Math.} \textbf{293} (1981), no.~1, 59--61.

\bibitem{gauduchon-mathann}
P. Gauduchon, La $1$-forme de torsion d'une variété hermitienne compacte, {\em Math. Ann.} \textbf{267} (1984), no.~4, 495--518.

\bibitem{gauduchon-bumi}
P. Gauduchon, Hermitian connections and Dirac operators, {\em Boll. Un. Mat. Ital. B (7)} \textbf{11} (1997), no.~2, suppl., 257--288.

\bibitem{gauduchon-book}
P. Gauduchon, {\em Calabi's extremal K\"ahler metrics: An elementary introduction}.

\bibitem{kielhofer}
Kielhöfer, Hansjörg Bifurcation theory. An introduction with applications to PDEs. Applied Mathematical Sciences, 156. Springer-Verlag, New York, 2004.

\bibitem{lee-parker}
J.~M. Lee, T.~H. Parker, The Yamabe problem, {\em Bull. Amer. Math. Soc. (N.S.)} \textbf{17} (1987), no.~1, 37--91.

\bibitem{lelong-ferrand}
J. Lelong-Ferrand, Transformations conformes et quasi-conformes des varietés riemanniennes (démonstration de la conjecture de A. Lichnerowicz), {\em Acad. Roy. Belgique Sci. Mem. Coll. (2)} \textbf{39} (1971), no.~5, 44 pp..

\bibitem{liu-yang}
K.-F. Liu, X.-K. Yang, Geometry of Hermitian manifolds, {\em Internat. J. Math.} \textbf{23} (2012), no.~6, 1250055, 40 pp..

\bibitem{liu-yang-2}
K.-F. Liu, X.-K. Yang, Ricci curvatures on Hermitian manifolds, \texttt{arXiv:1404.2481v3 [math.DG]}.

\bibitem{malchiodi}
A. Malchiodi, Liouville equations from a variational point of view, to appear in {\em Proceedings ICM}. 

\bibitem{michelsohn}
M.~L. Michelsohn, On the existence of special metrics in complex geometry, {\em Acta Math.} \textbf{149} (1982), no.~3-4, 261--295.

\bibitem{obata}
M. Obata, Conformal transformations of Riemannian manifolds, {\em J. Diff. Geom.} \textbf{4} (1970), no.~3, 311--333.

\bibitem{schoen}
R. Schoen, Conformal deformation of a Riemannian metric to constant scalar curvature, {\em J. Diff. Geom.} \textbf{20} (1984), no.~2, 479--495.

\bibitem{smoller-wasserman}
J. Smoller, A. G. Wasserman, Bifurcation and symmetry-breaking. {\em Invent. Math.} \textbf{100} (1990), no. 1, 63–95.

\bibitem{teleman}
A. Teleman, The pseudo-effective cone of a non-Kählerian surface and applications, {\em Math Ann.} \textbf{335} (2006), no. 4, pp. 965--989.

\bibitem{tosatti}
V. Tosatti, Non-Kähler Calabi-Yau manifolds, to appear in {\em Contemp. Math.}, \texttt{arXiv:1401.4797 [math.DG]}.

\bibitem{tosatti-weinkove}
V. Tosatti, B. Weinkove, The complex Monge-Ampère equation on compact Hermitian manifolds, {\em J. Amer. Math. Soc.} \textbf{23} (2010), no.~4, 1187--1195.

\bibitem{trudinger}
N.~S. Trudinger, Remarks concerning the conformal deformation of Riemannian structures on compact manifolds, {\em Ann. Scuola Norm. Sup. Pisa (3)} \textbf{22} (1968), no.~2, 265--274.

\bibitem{westreich}
D. Westreich, Bifurcation at eigenvalues of odd multiplicity, {\em Proc. Amer. Math. Soc.} \textbf{41} (1973), no.~2, 609--614.
 
\bibitem{yamabe}
H. Yamabe, On a deformation of Riemannian structures on compact manifolds, {\em Osaka Math. J.} \textbf{12} (1960), no.~1, 21--37.

\end{thebibliography}
\end{document}